\documentclass[11pt]{article}

\usepackage[textwidth=17cm,textheight=24cm]{geometry}
\usepackage{enumerate}
\usepackage[mathscr]{eucal}
\usepackage{amsmath,amssymb,amsthm,mathtools}
\usepackage[numbers]{natbib}
\usepackage{hyperref,nameref}

\newtheorem{lemma}{Lemma}
\newtheorem{proposition}{Proposition}
\newtheorem{corollary}{Corollary}
\newtheorem{theorem}{Theorem}

\newenvironment{smatrix}[1][r]{\left(\begin{smallmatrix*}[#1]}{\end{smallmatrix*}\right)}

\usepackage[usenames,dvipsnames]{xcolor}

\makeatletter
\renewcommand\subsubsection{\@startsection{subsubsection}{3}{\z@}%
  {-3.25ex\@plus -1ex \@minus -.2ex}%
  {-1.5ex \@plus .2ex}%
  {\normalfont\normalsize\bfseries}}

\let\orgdescriptionlabel\descriptionlabel
\renewcommand*{\descriptionlabel}[1]{%
  \let\orglabel\label
  \let\label\@gobble
  \phantomsection
  \edef\@currentlabel{#1}%
  \let\label\orglabel
  \orgdescriptionlabel{\bfseries #1}%
}
\makeatother

\newcommand{\1}{{\bf 1}} 

\DeclareMathOperator{\Dirac}{Dirac}
\DeclareMathOperator{\expit}{expit}
\DeclareMathOperator{\logit}{logit}
\DeclareMathOperator{\argmin}{argmin}
\newcommand{\calR}{\mathcal{R}}

\newcommand{\xA}{{\cal A}}
\newcommand{\xF}{{\cal F}}
\newcommand{\xG}{{\cal G}}
\newcommand{\pG}{\mathbb{G}_n^{h}}
\newcommand{\lG}{\mathbb{G}^{h}}

\newcommand{\xM}{\mathscr{M}}
\newcommand{\xMu}{{\cal M}}

\newcommand{\xP}{\mathbb{P}}
\newcommand{\xQ}{{\cal Q}}
\newcommand{\xR}{\mathbb{R}}
\newcommand{\xS}{\mathbb{S}}
\newcommand{\xT}{\mathcal{T}}
\newcommand{\T}{T_{N}}

\newcommand{\xW}{{\cal W}}

\newcommand{\xV}{{\cal V}}
\newcommand{\xX}{{\cal X}}

\newcommand{\xO}{{\cal O}}

\newcommand{\PHT}{P_{R_{N}}^{\boldsymbol p}}
\newcommand{\PHTW}{P_{R_{N},W}^{\boldsymbol p}}
\newcommand{\PHTX}{P_{R_{N},X}^{\boldsymbol p}}

\newcommand{\prob}[1]{P\left(#1\right)} 
\newcommand{\probi}[2]{P_{#1}\left(#2\right)} 
\newcommand{\esp}[1]{E\left[#1\right]} 
\newcommand{\espi}[2]{E_{#1}\left[#2\right]} 
\newcommand{\vari}[2]{\text{Var}_{#1}\left(#2\right)} 

\allowdisplaybreaks

\begin{document}
\title{Practical targeted learning from large data sets by survey sampling}
\author{P. Bertail, A. Chambaz, E. Joly\\
  Modal'X, Universit\'e Paris Ouest Nanterre}
\maketitle

\begin{abstract}
  We address the practical construction of asymptotic confidence intervals for
  smooth  ({\em  i.e.},   pathwise  differentiable),  real-valued  statistical
  parameters by targeted learning from independent and identically distributed
  data in contexts  where sample size is so large  that it poses computational
  challenges.   We observe  some  summary measure  of all  data  and select  a
  sub-sample from  the complete  data set by  Poisson rejective  sampling with
  unequal  inclusion probabilities  based on  the summary  measures.  Targeted
  learning is carried  out from the easier to handle  sub-sample.  We derive a
  central limit theorem  for the targeted minimum loss  estimator (TMLE) which
  enables  the  construction  of  the  confidence  intervals.   The  inclusion
  probabilities  can be  optimized to  reduce the  asymptotic variance  of the
  TMLE.  We illustrate the procedure with two examples where the parameters of
  interest  are  variable  importance  measures  of  an  exposure  (binary  or
  continuous) on an outcome. We also conduct a simulation study and
  comment on its results. \\~\\
  keywords: semiparametric  inference; survey sampling; targeted  minimum loss
  estimation (TMLE) \\~\\
\end{abstract}


\section{Introduction}
\label{sec:intro}


Large  data  sets  are  ubiquitous  nowadays.   They  pose  computational  and
theoretical challenges.   We consider the  particular problem of  carrying out
inference     based      on     semiparametric     models      by     targeted
learning~\cite{vanderLaan06,TMLEbook} from  large data  sets.  We  mainly deal
with the fact  that the sample size $N$  is, say, huge.  Even if  we also take
advantage of  easy to handle summary  measures of the observations,  we do not
consider the specific difficulties yielded by  the messiness of real big data.
This is why we use the expression ``large data sets'' instead of ``big data''.

Confronted with large  data sets, many learning algorithms  fail to provide an
answer in a reasonable time if at all.  Following \cite{bertail:hal-00989585},
we overcome this computational limitation by {\em (i)} selecting $n$ among $N$
observations  with  unequal probabilities  and  {\em  (ii)} adapting  targeted
learning from this smaller, tamed data set.

Specifically,  our objective  is to  enable the  construction of  a confidence
interval   with   given  asymptotic   level   for   a  statistical   parameter
$\psi_{0}\equiv \Psi(P_{0})$  based on  a sample $O_{1},  \ldots, O_{N}$  of a
(huge number) $N$ of  independent and identically distributed (i.i.d.)  random
variables drawn  from $P_{0}  \in \xM$,  where $\Psi:\xM \to  \xR$ maps  a set
$\xM$  of measures  including possible  distributions of  $O_{1}$ to  the real
line.   We focus  on the  case that  the functional  $\Psi$ is  smooth  in the
following  sense.   For every  $P  \in  \xM$, there  exists  a  wide class  of
one-dimensional   paths  $\{P_{t}  :   t  \in   ]-c,c[\}  \subset   \xM$  with
$P_{t}|_{t=0} =  P$ and  an influence function  $D(P) \in L_{0}^{2}  (P)$ such
that, for all $|t| < c$,
\begin{eqnarray}
  \label{eq:linear:exp}
  \Psi(P_{t}) &  = & \Psi(P)  + \int D(P)  (dP_{t} - dP) +  o(t) \\
  \notag 
  & = & \Psi(P) + \int D(P) dP_{t} + o(t). 
\end{eqnarray}
Here,  we denote  $L_{0}^{2} (P)$  the set  of centered  and square-integrable
measurable functions relative to $P$.

Condition (\ref{eq:linear:exp}) trivially holds when $\Psi$ is linear. If, for
instance, $\Psi$  is given by $\Psi(P)  \equiv \int f dP$  for some measurable
function  $f$ integrable with  respect to  (wrt) all  elements of  $\xM$, then
(\ref{eq:linear:exp})  holds  with $D(P)  \equiv  f  -  \Psi(P)$ (without  the
$o$-term).  Even  in the  very simple  example where $f$  is the  identity and
$\xM$ consists of probability measures, hence $\Psi(P) = \espi{P}{O}$,
it may be computationally difficult,  if not impossible, to build a confidence
interval  for  $\psi_{0}=\Psi(P_{0})$  {\em  using all  observations},  merely
because it  may be  very challenging  to {\em access}  to all  of them  in the
context of large data sets. 

Typical  examples  of  functionals  satisfying  (\ref{eq:linear:exp})  include
pathwise        differentiable         functionals        as        introduced
in~\citep[][Section~25.3]{vdV98}.    We  will  give   two  examples   of  such
functionals.   Pathwise  differentiability   differs  slightly  from  Gateaux,
Hadamard and Fr\'echet differentiability. It is  one the of key notions in the
theory of semiparametric inference.

We  overcome  the  computational  hurdle  by  resorting  to  survey  sampling,
specifically  to rejective  sampling based  on Poisson  sampling with  unequal
inclusion  probabilities.   It  is  a  particular  case  of  sampling  without
replacement (we refer to~\cite{hanif1980sampling}  for an overview on sampling
without replacement). Survey sampling can also  rely on the so called sampling
entropy  \cite{berger1998rate,brewer2003high,grafstrom2010entropy}, but  we do
not  follow this  path. Also  known  as Sampford  sampling, rejective  Poisson
sampling  has been  thoroughly studied  for the  last five  decades since  the
publication  of  the  seminal  articles~\cite{hajek1964,sampford1967sampling}.
The key  object in the analysis  of Sampford sampling is  the Horvitz-Thompson
(HT) empirical  measure.  Asymptotic normality  of estimators based on  the HT
empirical  measure was  first  established  in~\cite{hajek1964}. A functional version for the cumulative distribution function was obtained by \citep{wang2012sample} . Our  analysis hinges on the recent  study of the HT empirical measure  from the viewpoint of
empirical  processes  theory  carried out  in~\cite{bertail:hal-00989585}  (we
refer the reader to this article for additional references).

For instance \citep{cardot2013confidence,CARDEGOJOLA2013} show practically how to implement confidence bands for model-assisted estimators of the mean when the variable of interest is functional and storage capacities are limited (with applications to electricity consumption curves). In that case, survey sampling techniques are interesting alternative to signal compression techniques. 

The joint use of survey sampling techniques in conjunction with semiparametric
models for inference  is not new~\cite{breslow2007weighted,Breslow08}.  To the
best of our knowledge, however, this is the first attempt to take advantage of
survey sampling to enable targeted learning when the data set is so large that
computational problems arise. In contrast to naive sub-sampling, sampling designs with unequal
probabilities offer a control over the efficiency of estimators. In this  light, we propose an alternative to the
so called online version of targeted learning~\cite{onlineTMLE}.

\subsubsection*{Organization.}

Section~\ref{sec:main} presents our procedure  for practical targeted learning
from large  data sets by survey  sampling and the central  limit theorem which
enables the  construction of confidence  intervals. Section~\ref{sec:examples}
illustrates Section~\ref{sec:main} with two  examples, where the parameters of
interest are variable importance measures of a (binary or continuous) exposure
on an outcome.   Section~\ref{sec:sim} summarizes the results  of a simulation
study. The proofs are given in appendix.

\section{Practical targeted learning}
\label{sec:main}

Throughout  the   article,  we  denote  $\mu   f  \equiv  \int   f  d\mu$  and
$\|f\|_{2,\mu} \equiv  (\mu f^{2})^{1/2}$ for  any measure $\mu$  and function
$f$ (measurable and integrable wrt $\mu$).


\subsection{Survey sampling  from the large  data set and construction  of the
  estimator}
\label{sec:surv:samp}

\subsubsection*{Rejective sampling.}

Let $n(N)$ be a deterministic,  user-supplied number of observations to select
by survey sampling.  It is a practical, computationally  tractable sample size
as opposed to  the unpractical, huge $N$.  Because  our results are asymptotic
we impose that, as $N \to \infty$,
\begin{equation*}
  n (N) \to \infty \quad \text{and} \quad \frac{n(N)}{N} \to 0.  
\end{equation*}
In the rest of this article, we will simply denote $n$ for $n(N)$.  

We   employ  a  specific   survey  sampling   scheme  called   {\em  rejective
  sampling}~\cite{hajek1964,bertail:hal-00989585}.   The  random selection  of
observations  from the  complete  data  set can  depend  on easily  accessible
summary measures  $V_{1}, \ldots, V_{N}  \in \xV$ attached to  $O_{1}, \ldots,
O_{N}$. Typically, $V_{1}, \ldots,  V_{N}$ take finitely many different values
or   are  low-dimensional,  and   the  implementation   of  the   database  is
structured/organized based on the values of $V_{1}, \ldots, V_{N}$.

Let $h$ be  a (measurable) function on $\xV$ such  that $h(\xV) \subset [c(h),
\infty)$ for some constant $c(h)>0$. For each $1 \leq i \leq N$, define
\begin{equation*}
  p_i \equiv \frac{nh(V_i)}{N}.
\end{equation*}
For $N$ large enough, $p_{1}, \ldots, p_{N} \in (0,1)$. Introduce
\begin{itemize}
\item   $\varepsilon_{1},   \ldots,   \varepsilon_{N}$  independently   drawn,
  conditionally on  $V_{1}, \ldots,  V_{N}$, from the  Bernoulli distributions
  with parameters $p_{1}, \ldots, p_{N}$, respectively;
\item $(\eta_{1},  \ldots, \eta_{N})$ drawn, conditionally  on $V_{1}, \ldots,
  V_{N}$,  from  the conditional  distribution  of $(\varepsilon_{1},  \ldots,
  \varepsilon_{N})$ given $\sum_{i=1}^{N} \varepsilon_{i} = n$.
\end{itemize}

The  subset of  $n$ observations  randomly selected  by rejective  sampling is
$\{O_{i} : \eta_{i}=1, 1 \leq i \leq N\}$.
It is associated with the so-called HT empirical measure defined by
\begin{equation}
  \label{eq:PHT}
  \PHT \equiv \frac{1}{N}\sum_{i=1}^N \frac{\eta_i}{p_i}\Dirac(O_i).
\end{equation}
Note that $\PHT$  is not necessarily a probability  measure.  However, if
$h \equiv 1$ then $\PHT$ is a probability measure, and rejective sampling
is equivalent to selecting $n$ observations among $O_1,\dots,O_N$ uniformly.

For   computational   reasons,   it   is   not  desirable   that   the   event
``$\sum_{i=1}^{N}   \varepsilon_{i}  =  n$''   be  too   unlikely.   Lemma~3.1
in~\citep{hajek1964}  shows  that the  conditional  probability  of the  event
``$\sum_{i=1}^{N}  \varepsilon_{i} =  k$'' is  maximized when  $k$  equals the
conditional expectation of $\sum_{i=1}^{N} \varepsilon_{i}$, in which case the
conditional probability is asymptotically equivalent to $(2 \pi \sum_{i=1}^{N}
p_{i}  (1-p_{i}))^{-1/2}$.   Because the  conditional  expectation of  $n^{-1}
\sum_{i=1}^{N}  \varepsilon_{i}$ equals  $n^{-1}\sum_{i=1}^{N} p_{i}  = N^{-1}
\sum_{i=1}^{N}   h(V_{i})$,   which   converges   $P_{0}$-almost   surely   to
$\espi{P_{0}}{h(V)}$, it is thus good  practice to choose function $h$ in such
a way  that $\espi{P_{0}}{h(V)}$ be close  to 1.  When  $V_{1}, \ldots, V_{N}$
take  finitely  many different  values,  it  is  easy to  estimate  accurately
$\espi{P_{0}}{h(V)}$ on an independent sample  and, therefore, to adapt $h$ so
that $\espi{P_{0}}{h(V)} \approx 1$.

\subsubsection*{Practical, targeted estimator.}

Assume that we have constructed  $P_{n}^{*} \in \xM$ targeted to $\psi_{0}$ in
the sense that
\begin{equation}
  \label{eq:targeted}
  \PHT D(P_{n}^{*}) = o_{P} (1/\sqrt{n}).
\end{equation}
We define $\psi_{n}^{*} \equiv \Psi(P_{n}^{*})$ as our substitution estimator.
This  construction  frames $\psi_{n}^{*}$  in  the  paradigm  of the  targeted
minimum loss estimation methodology~\citep{vanderLaan&Rubin06,TMLEbook}.

\subsection{Main theorem}
\label{sec:theorem}

Consider a class  $\xF$ of functions mapping a measured  space $\xX$ to $\xR$.
Set   $\delta  >   0$  and   a  semi-metric   $d$  or   a  norm.    We  denote
$N(\varepsilon, \xF,d)$  the $\varepsilon$-covering  number of $\xF$  wrt $d$,
{\em i.e.}, the minimum number of  $d$-balls of radius $\varepsilon$ needed to
cover  $\xF$.   The corresponding  entropy  integral  for $\xF$  evaluated  at
$\delta$                                                                    is
$J(\delta, \xF,  d) \equiv  \int_{0}^{\delta} \sqrt{\log N(\varepsilon  , \xF,
  d)}d\varepsilon$.

Let $\calR: \xM^{2} \to \xR$ be given by
\begin{equation}
  \label{eq:second:order}
  \calR(P, P') \equiv \Psi(P') - \Psi(P) - \int D(P) (dP' - dP)
\end{equation}
where the  influence function $D(P)$ is  defined before (\ref{eq:linear:exp}).
The real number $\calR(P_{n}^{*}, P_{0})$ can be interpreted as a second-order
term in an  expansion of $\psi_{n}^{*} = \Psi(P_{n}^{*})$  around $P_{0}$.  By
(\ref{eq:linear:exp}),   we   focus    on   functionals   $\Psi$   such   that
$\calR(P,P_{t}) = o (t)$ for a  wide class of one-dimensional paths $\{P_{t} :
t \in  ]-c,c[\} \subset \xM$ such  that $P_{t}|_{t=0} = P$.  This statement is
clarified in the examples of Section~\ref{sec:examples}.

We suppose the existence of $\xF  \subset \{D(P) : P \in \xM\}$ satisfying the
three following assumptions:
\begin{description}
\item[A1  (complexity)]  $\xF$ is  separable,  for  every  $f\in \xF$,  $P_{0}
  f^{2}h^{-1} < \infty$, and $J(1, \xF, \|\cdot\|_{2,P_{0}}) < \infty$.
\item[A2 (uniform  convergence of empirical metric)] For  every $f,f'\in \xF$,
  if
  \begin{equation}
    \label{eq:rhoN}
    \rho_{N}^2(f,f')\equiv\frac{1}{N}\sum_{i=1}^N (f(O_i)-f'(O_i))^2
  \end{equation}
  then, $P_{0}$-almost surely,
  \[
  \sup_{f,f'\in\xF}      \left|\frac{\rho_N(f,f')}{\|f-f'\|_{2,P_{0}}}-1\right|
  \underset{N \to \infty}{\longrightarrow} 0.
  \]
\item[A3  (first order  convergence)] With  $P_{0}$-probability tending  to 1,
  $D(P_{n}^{*})  \in  \xF$,  and  there  exists  $f_{1}  \in  \xF$  such  that
  $\|D(P_{n}^{*})  - f_{1}\|_{2,P_{0}}  = o_{P}  (1)$. Moreover,  one  knows a
  conservative   estimator  $\Sigma_{n}$   of  $\sigma_{1}^{2}   \equiv  P_{0}
  f_{1}^{2}h^{-1}$.
\end{description}
Under {\bf A1}, we can define $\Sigma : \xF^{2} \to \xR$ given by
\begin{equation}
  \label{eq:Sigma} 
  \Sigma(f,f')  \equiv  P_{0}    ff'h^{-1}.
\end{equation}
In particular, $\sigma_{1}^{2}$ in {\bf A3} equals $\Sigma(f_{1}, f_{1})$.  An
additional assumption is needed:
\begin{description}
\item[A4  (second order  term)]  There exists  a  real-valued random  variable
  $\gamma_{n}$ converging in probability to  $\gamma_{1} \neq 1$ and such that
  $\gamma_{n}  (\psi_{n}^{*} - \psi_{0})  + \calR  (P_{n}^{*}, P_{0})  = o_{P}
  (1/\sqrt{n})$.   Moreover, one  knows  an estimator  $\Gamma_{n}$ such  that
  $\Gamma_{n} - \gamma_{n} = o_{P}(1)$.
\end{description}
We can now state our main theorem.
\begin{theorem}
  \label{thm:main}
  Assume that {\bf A1}, {\bf A2}, {\bf A3} and {\bf A4} are met. Then it holds
  that $(1-\gamma_{n}) \sqrt{n} (\psi_{n}^{*} - \psi_{0})$ converges in law to
  the   centered  Gaussian   distribution   with  variance   $\sigma_{1}^{2}$.
  Consequently, for any $\alpha \in (0,1)$,
  \begin{equation*}
    \left[\psi_{n}^{*}           \pm                    
      \frac{\xi_{1-\alpha/2} \sqrt{\Sigma_{n}}}{(1-\Gamma_n)\sqrt{n}} \right] 
  \end{equation*}
  is a confidence interval with asymptotic coverage no less than $(1-\alpha)$.
\end{theorem}

\subsubsection*{Comments.}

Assumption  {\bf A1}  is typical  in semiparametric  inference, and  should be
interpreted    as    a    constraint    on   the    complexity    of    $\xF$.
Theorem~\ref{thm:main} relies on the  convergence of an empirical process, see
Theorem~\ref{thm:result_sondage}.             The           proof           of
Theorem~\ref{thm:result_sondage} uses a chaining argument, and {\bf A2} allows
to upper-bound the resulting {\em random} term $J(\delta, \xF, \rho_{N})$ by a
{\em  deterministic} term  $J(\delta, \xF,  \|\cdot\|_{2,P_{0}})$.
We say  that a class $\mathcal{C}$  has finite uniform entropy  integral if it
admits an envelope function $F$ and
\begin{equation*}
  \int_{0}^{\infty} \sup_{\rho} \sqrt{ \log N(\epsilon \|F\|_{2, \rho}, \mathcal{C},
    \|\cdot\|_{2,\rho})} d\epsilon < \infty,
\end{equation*}
where the supremum is over all  probability measures $\rho$ on $\xO$ such that
$\|F\|_{2,\rho} > 0$.  Assumption {\bf A2} can be replaced by the alternative
\begin{description}
\item[A2*] The class $\xF$ has a finite uniform entropy integral.
\end{description}
VC-classes    of    uniformly   bounded    functions    satisfy   {\bf    A2*}
\cite[][Section~2.6]{van1996weak}.   Finally,  {\bf   A3}  and  {\bf  A4}  are
technical conditions  required by  the TMLE procedure.   The former is  not as
mild as one  may think at first sight, because  the conservative estimation of
$\sigma_{1}^{2}$  is not  trivial.   For  instance, it  is  not guaranteed  in
general that the substitution estimator
\begin{equation}
  \label{eq:Sigma:n}
  \Sigma_{n}   \equiv    \PHT   D(P_{n}^{*})^2h^{-1}
\end{equation}
estimates  conservatively  $\sigma_{1}^{2}$.   Relying on  the  non-parametric
bootstrap is not a solution either in general.  

We argued  that $\calR(P_{n}^{*},  P_{0})$ should be  interpreted as  a second
order term.  In the simplest examples, this is literally the case and assuming
$\calR(P_{n}^{*}, P_{0})  = o_{P} (1/\sqrt{n})$  is natural, see  for instance
Section~\ref{subsec:binary}.   Sometimes,  $\calR(P_{n}^{*},  P_{0})$ must  be
corrected by adding $\gamma_{n} (\psi_{n}^{*}  - \psi_{0})$ so that it becomes
natural to assume  that the corrected expression is  $o_{P} (1/\sqrt{n})$, see
for instance Section~\ref{subsec:continuous}.

and {\bf A4}
is met  with $\gamma_{n} =  0$, see for  instance Section~\ref{subsec:binary}.
Allowing  $\gamma_{n}$   to  differ  from   0  gives  more   flexibility.   In
Section~\ref{sec:examples},  we give  additional conditions  which  imply {\bf
  A4}.

Knowing  the asymptotic  variance of  $(1-\gamma_{n})\sqrt{n}  (\psi_{n}^{*} -
\psi_{0})$ allows to discuss further the choice of $h$.  Introduce 
\begin{equation}
  \label{eq:f2}
  f_2(V)  \equiv \sqrt{\espi{P_0}{f_1(O)^2\middle|V}},
\end{equation}
which  satisfies $\sigma_{1}^{2}  = P_0  f_1^2 h^{-1}  =  P_0f_2^2h^{-1}$. The
Cauchy-Schwarz inequality yields
\begin{equation}
  \label{eq:CS}
  (P_0 f_2)^2 \le P_0 f_2^2h^{-1} \times P_0 h = \sigma_{1}^{2} \times P_{0} h,
\end{equation}
and equality  occurs when $f_2$ and  $h$ are linearly  dependent. Moreover, it
should hold that $P_{0}  h = 1$. In view of (\ref{eq:CS}),  the optimal $h$ is
$f_2/P_0 f_2$, assuming  that $P_{0} f_{2} > 0$  (otherwise, $\sigma_{1}^{2} =
0$).  This  argument neglects the  second-order dependence of  $\gamma_{n}$ on
$h$. In practice,  we would first sample $n_{0}$ data  using $h_{0} \equiv 1$,
use  them  to  estimate  $f_{2}$  and $P_{0}  f_{2}$  with  $f_{2,n_{0}}$  and
$Z_{2,n_{0}}$, then  finally define $h  \equiv f_{2,n_{0}} /  Z_{2,n_{0}}$ and
exclude the sampled data from $\{O_{1}, \ldots, O_{N}\}$.

The following expansion taken from  the proof of Theorem \ref{thm:main} partly
explains    why    $\sigma_{1}^{2}$   is    the    asymptotic   variance    of
$(1-\gamma_{n})\sqrt{n}    (\psi_{n}^{*}    -    \psi_{0})$:    denoting    by
$P_{0}^{\varepsilon}$  the shared  distribution of  $(O_{1}, \varepsilon_{1}),
\ldots, (O_{N}, \varepsilon_{N})$, it holds for any $f$ in $\xF$ that
\begin{eqnarray*}
  \vari{P_{0}^{\varepsilon}}{\frac{1}{N}\sum_{i=1}^N
    \frac{f(O_i)\varepsilon_i}{p_i}} & = & \frac{1}{N}
  \vari{P_{0}^{\varepsilon}}{\frac{f(O_{1})\varepsilon_{1}}{p_{1}}}\\
  & = & \frac{1}{N} \left( 
    \espi{P_{0}}{f^2(O_{1})\left(\frac{1}{p_{1}}-1\right)}      +
    \vari{P_{0}}{f(O)} \right).
\end{eqnarray*}
If,  contrary to  facts,  we  could take  $p_{1}\equiv  1$ (or,  equivalently,
$n\equiv N$  and $h\equiv 1$), then  the asymptotic variance  of the resulting
TMLE estimator would be of  the form $N^{-1}\vari{P_{0}}{f(O)}$ for some limit
$f$, as typically expected.   In Section~\ref{sec:surv:samp} $p_{1}$ is chosen
in such  a way that $1/p_{1}$  is typically much larger  than~1. Actually, the
above RHS expression at $f \equiv f_{1}$ rewrites
\begin{equation}
  \label{eq:decomp}
  \frac{1}{n}          \left(P_{0}         f_{1}^{2}          h^{-1}         +
    \frac{n}{N}(P_{0}f_{1})^{2}\right) = 
  \frac{1}{n} \left(\sigma_{1}^{2} + o(1)\right). 
\end{equation}
Note the absence of a centering term in $P_{0} f_1^{2} h^{-1}$.

\section{Two examples}
\label{sec:examples}

We  illustrate  Theorem~\ref{thm:main}  with  the inference  of  two  variable
importance     measures    of     an    exposure,     either     binary,    in
Section~\ref{subsec:binary},            or            continuous,           in
Section~\ref{subsec:continuous}.   In  both  examples, the  $i$th  observation
$O_{i}$ writes  $(W_{i}, A_{i},  Y_{i}) \in \xO  \equiv \xW \times  \xA \times
[0,1]$. Here,  $W_{i} \in \xW$  is the $i$th  context, $A_{i} \in \xA$  is the
$i$th exposure  and $Y_{i}  \in [0,1]$  is the $i$th  outcome.  In  the binary
case, $\xA \equiv \{0,1\}$.  In the  continuous case, $\xA \ni 0$ is a bounded
subset of $\xR$  containing 0, which serves as a  reference level of exposure.
Typically,  in biostatistics or  epidemiology, $W_{i}$  could be  the baseline
covariate describing the $i$th  subject, $A_{i}$ could describe her assignment
({\em e.g.}, treatment or placebo when $\xA = \{0,1\}$ or dose-level when $\xA
\subset \xR$) or exposure ({\em e.g.},  exposed or not when $\xA = \{0,1\}$ or
level  of exposure when  $\xA \subset  \xR$), and  $Y_{i}$ could  quantify her
biological response.

\subsection{Variable importance measure of a binary exposure}
\label{subsec:binary}

In this section, $\xA \equiv \{0,1\}$ and $\psi_{0}$ equals
\begin{equation}
  \label{eq:psi0:bin}
  \psi_{0}^{b}   \equiv  \espi{P_{0}}{\espi{P_{0}}{Y   \middle|   A=1,  W}   -
    \espi{P_{0}}{Y \middle| A=0, W}}
\end{equation}
(the superscript ``$b$'' stands for ``binary'').  Now, let $\xM$ be the subset
of the set of finite measures  on $\xO \equiv \xW \times \{0,1\} \times [0,1]$
equipped with the  Borel $\sigma$-field such that every $P  \in \xM$ puts mass
on all events of the form  $B_{1} \times \{a\} \times B_{2}$ ($a=0,1$, $B_{1}$
and $B_{2}$  Borel sets  of $\xW$ and  $[0,1]$).  It  contains the set  of all
possible data-generating  distributions for $O_{1}$ such  that the conditional
distribution of  $A$ given $W$  is not deterministic, including  $P_{0}$.  For
each $P  \in \xM$, we denote  $P_{W}$, $P_{A|W}$ and  $P_{Y|A,W}$ the marginal
measure of $W$ and conditional measures  of $A$ and $Y$ given $W$ and $(A,W)$,
respectively.   (The  conditional  measure  $P_{A|W}$ is  $P(\xO)$  times  the
conditional  law   of  $A$  given  $W$  under   the  probability  distribution
$P/P(\xO)$.  The conditional measure  $P_{Y|A,W}$ is defined analogously.)  We
see  $\psi_{0}^{b}$ as  the  value  at $P_{0}$  of  the functional  $\Psi^{b}$
characterized over $\xM$ by
\begin{equation}
  \label{eq:Psi:bin}
  \Psi^{b}                (P)                \equiv               \int_{\xW}
  \left(\int_{[0,1]} y \left(dP_{Y|A=1,W=w} (y) - dP_{Y|A=0,W=w}(y)\right)
  \right) dP_W (w).
\end{equation}

In particular,  if $P$  is a possible  data-generating {\em  distribution} for
$O_{1}$ ({\em i.e.}, if $P(\xO)=1$), then
\begin{equation*}
  \Psi^{b} (P) = \espi{P}{\espi{P}{Y \middle| A=1, W} -
    \espi{P}{Y \middle| A=0, W}}. 
\end{equation*}
Moreover,  under   additional  causal  assumptions,  $\Psi^{b}   (P)$  can  be
interpreted as  the additive  causal effect of  the exposure on  the response,
see~\citep{Pearl00,TMLEbook}.

Two infinite-dimensional features of every  $P \in \xM$ will play an important
role in the analysis.  Namely, for each  $P \in \xM$ and $(w,a) \in \xW \times
\xA$, we  introduce and denote  $g_{P} (0|w)\equiv P_{A|W=w}  (\{0\})$, $g_{P}
(1|w)  \equiv  P_{A|W=w} (\{1\})$,  and  $Q_{P}  (a,w)  \equiv \int_{[0,1]}  y
dP_{Y|A=a,W=w}(y)$.   In  particular  if   $P(\xO)=1$,  then  $g_{P}  (1|W)  =
P(A=1|W)$ is  the conditional probability  that the binary exposure  equal one
and  $Q_{P} (A,W)  = \espi{P}{Y|A,W}$  is the  conditional expectation  of the
response given exposure and context.

\subsubsection*{Pathwise differentiability.}

The functional $\Psi^{b}$  is pathwise differentiable at each  $P \in \xM$ wrt
the maximal tangent space $L_{0}^{2} (P)$ (the space of functions $s : \xO \to
\xR$  such  that  $Ps  =  0$   and  $Ps  ^{2}  <  \infty$)  in  the  following
sense~\citep[][Chapter~5 and Section~A.3]{TMLEbook}:
\begin{lemma}
  \label{lem:pathwise:bin}
  Fix $P  \in \xM$ and introduce  the influence curve  $D^{b}(P) \in L_{0}^{2}
  (P)$ given by $D^{b}(P) \equiv D_{1}^{b} (P) + D_{2}^{b} (P)$ with
  \begin{eqnarray*}
    D_{1}^{b}(P) (O) & \equiv & Q_{P}(1,W) - Q_{P}(0,W) - \Psi^{b}(P),\\
    D_{2}^{b}(P) (O) & \equiv & (Y - Q_{P}(A,W)) \frac{2A-1}{g_{P}(A|W)}.
  \end{eqnarray*}
  For  every  uniformly  bounded  $s  \in  L_{0}^{2} (P)$  and  every  $t  \in
  ]-\|s\|_{\infty}^{-1},  \|s\|_{\infty}^{-1}[$, define  $P_{s,t} \in  \xM$ by
  setting
  \begin{equation*}
    \frac{dP_{s,t}}{dP} = 1 + t s.
  \end{equation*}
  It holds  that $t \mapsto \Psi^{b}  (P_{s,t})$ is differentiable at  0 (as a
  function  from  $\xR$  to  $\xR$)  with  a  derivative  at  0  equal  to  $P
  D^{b}(P)s$. 

  The asymptotic  variance of any  regular estimator of $\Psi^{b}  (P_{0})$ is
  larger than the Cram\'er-Rao lower-bound $P_{0} D^{b}(P_{0})^{2}$. Moreover,
  for any $P,P' \in \xM$,
  \begin{equation}
    \label{eq:double:rob:bin}
    P D^{b}(P') = \Psi^{b} (P) - \Psi^{b} (P') + P (2A-1) (Q_{P'} - Q_{P})
    \left(\frac{1}{g_{P}} - \frac{1}{g_{P'}} \right). 
  \end{equation}
  Consequently if  $P D^{b}  (P') =  0$, then $\Psi^{b}  (P') =  \Psi^{b} (P)$
  whenever $g_{P'} = g_{P}$ {\bf\em or} $Q_{P'} = Q_{P}$.
\end{lemma}
The last statement is called a ``double-robustness property''.  Let $\calR^{b}
: \xM^{2} \to \xR$ be given by
\begin{equation}
  \label{eq:second:order:bin}
  \calR^{b} (P, P') \equiv \Psi^{b} (P') - \Psi^{b} (P) - (P' - P)D^{b} (P), 
\end{equation}
as in (\ref{eq:second:order}). In particular,
\begin{eqnarray*}
  \calR^{b} (P,P_{s,t}) &  = & \Psi^{b} (P_{s,t}) - \Psi^{b}  (P) - (P_{s,t} -
  P) D^{b} (P) \\
  & = & \Psi^{b}  (P_{s,t}) - \Psi^{b} (P) - tP D^{b}  (P)s = o(t),
\end{eqnarray*}
showing that (\ref{eq:linear:exp}) is met.

Furthermore, (\ref{eq:double:rob:bin}) and $P D^{b} (P) = 0$ imply
\begin{equation*}
  \calR^{b} (P, P') = P' (2A-1) (Q_{P'} - Q_{P})
  \left(\frac{1}{g_{P'}} - \frac{1}{g_{P}} \right). 
\end{equation*}
In the context of this example,  {\bf A4} is fulfilled with $\gamma_{n} \equiv
0$ (hence $\Gamma_{n} \equiv 0$ and $\gamma_{1}=0$) when
\begin{equation}
\label{eq:verifRasympt}
  \calR^{b}  (P_{n}^{*}, P_{0})  =  P_{0} (2A-1)  (Q_{P_{n}^{*}} -  Q_{P_{0}})
  \left( \frac{1}{g_{P_{n}^{*}}} - \frac{1}{g_{P_{0}}}\right) = o_{P}
  (1/\sqrt{n}). 
\end{equation}
Through the product, we will  draw advantage of the synergistic convergences of
$Q_{P_{n}^{*}}$  to $Q_{P_{0}}$  and  $g_{P_{n}^{*}}$ to  $g_{P_{0}}$ (by  the
Cauchy-Schwarz inequality  for example).  Note  that if $g_{P_{0}}$  is known,
then we can impose that $g_{P_{n}^{*}} = g_{P_{0}}$ and $\calR^{b} (P_{n}^{*},
P_{0}) = 0$ exactly. 

\subsubsection*{Construction of the targeted estimator.}


Let $\xQ^{w}$ and $\xG^{w}$ be  two user-supplied classes of functions mapping
$\xA  \times \xW$  to $[0,1]$.   We impose  that the  elements of  $\xQ^w$ are
uniformly bounded away  from 0 and 1.  Similarly, we  impose that the elements
of $\xG^w$ are uniformly bounded away from 0.  Let $\ell$ be the logistic loss
function given by
\begin{equation*}
  -\ell(u,v) \equiv u \log (v) + (1-u) \log (1-v)
\end{equation*}
(all $u,v \in [0,1]$ with conventions $\log(0)=-\infty$ and $0 \log(0)=0$).

We first  estimate $Q_{P_0}$  and $g_{P_{0}}$ with  $Q_{n}$ and  $g_{n}$ built
upon $\PHT$, $\xQ^{w}$ and $\xG^{w}$.  For instance, one could simply minimize
(weighted) empirical risks and define
\begin{eqnarray*}
  Q_{n} & \equiv & \mathop{\argmin}_{Q \in \xQ^{w}} \PHT \ell(Y, Q(A,W))
  =   \mathop{\argmin}_{Q \in  \xQ^{w}}  \sum_{i=1}^{N} \frac{\eta_{i}}{p_{i}}
  \ell(Y_{i}, Q(A_{i},W_{i})),\\
  g_{n} & \equiv & \mathop{\argmin}_{g \in \xG^{w}} \PHT \ell(A, g(A|W))
  =  \mathop{\argmin}_{g \in  \xG^{w}}  \sum_{i=1}^{N} \frac{\eta_{i}}{p_{i}}
  \ell(A_{i}, g(A_{i}|W_{i}))
\end{eqnarray*}
(assuming  that  the  $\argmin$s  exist).   Alternatively,  one  could  prefer
minimizing  cross-validated (weighted)  empirical  risks. This  is beyond  the
scope of  this article but will be  studied in future work.   We also estimate
the marginal distribution $P_{0,W}$ of $W$ under $P_{0}$ with
\begin{equation}
  \label{eq:PHT:W}
  \PHTW \equiv \frac{1}{N}\sum_{i=1}^N \frac{\eta_i}{p_i}\Dirac(W_i).
\end{equation}
Let  $P_{n}^{0}$  be   a  measure  such  that  $Q_{P_{n}^{0}}   =  Q_{n}$  and
$P_{n,W}^{0} = \PHTW$. Then
\begin{equation}
  \label{eq:Psin0bis}
  \Psi^{b}    (P_{n}^{0})    =   \frac{1}{N}\sum_{i=1}^N    \frac{\eta_i}{p_i}
  \left(Q_{n} (1, W_{i}) - Q_{n} (0, W_{i})\right)
\end{equation}
is an estimator of $\psi_{0}^{b}$, whose construction is not tailored/targeted
to $\psi_{0}^{b}$.  It is now time to target the inference procedure.

Targeting the inference procedure consists  in modifying $P_{n}^{0}$ in such a
way that the resulting  $P_{n}^{*}$ satisfies (\ref{eq:targeted}) with $D^{b}$
substituted for $D$. We first note that, by construction of $P_{n}^{0}$, 
\begin{equation*}
  \PHT D_{1}^{b} (P_{n}^{0}) = \PHTW D_{1}^{b} (P_{n}^{0}) = 0. 
\end{equation*}
This equality is equivalent to (\ref{eq:Psin0bis}).  

The construction of $P_{n}^{*}$ based on $P_{n}^{0}$ reduces to ensuring $\PHT
D_{2}^{b}  (P_{n}^{*}) = o_{P}  (1/\sqrt{n})$.  We  achieve this  objective by
fluctuating the conditional  measure of $Y$ given $(A,W)$  only.  For this, we
introduce the  one-dimensional parametric  model $\{Q_{n} (t)  : t  \in \xR\}$
given by
\begin{equation*}
  \logit Q_{n} (t) (A,W) = \logit Q_{n} (A,W) + t \frac{2A-1}{g_{n} (A|W)}.  
\end{equation*}
This   parametric    model   fluctuates   $Q_{n}$   in    the   direction   of
$\frac{2A-1}{g_{n} (A|W)}$ in the sense that $Q_{n} (0) = Q_{n}$ and 
\begin{equation}
  \label{eq:score:bin}
  \frac{d}{dt}   \ell(Y,   Q_{n}   (t)    (A,W))   =   (Y   -   Q_{n}(t)(A,W))
  \frac{2A-1}{g_{n} (A|W)} 
\end{equation}
for all $t \in \xR$.  The optimal move along the fluctuation is indexed by
\begin{equation}
\label{eq:optShift}
  t_{n} \equiv \mathop{\arg\min}_{t \in \xR} \PHT \ell(Y, Q_{n} (t) (A,W))
\end{equation}
(note that  the random function $t  \mapsto \PHT \ell(Y, Q_{n}  (t) (A,W))$ is
strictly convex). 

Define $Q_{n}^{*} \equiv Q_{n} (t_{n})$ and let $P_{n}^{*}$ be any element $P$
of  $\xM$  such  that $Q_{P}  =  Q_{n}^{*}$,  $g_{P}  =  g_{n}$ and  $P_{W}  =
P_{n,W}^{0} = \PHTW$.  Our final estimator is 
\begin{equation*}
  \psi_{n}^{*} \equiv \Psi^{b} (P_{n}^{*}) = \frac{1}{N}\sum_{i=1}^N 
  \frac{\eta_i}{p_i}    \left(Q_{n}^{*}   (1,    W_{i})   -    Q_{n}^{*}   (0,
    W_{i})\right).
\end{equation*}
By  definition of $t_{n}$  and (\ref{eq:score:bin}),  we have  $\PHT D_{1}^{b}
(P_{n}^{*}) = 0$ (just like $\PHT D_{1}^{b} (P_{n}^{0}) = 0$) and
\begin{equation*}
  \PHT  \left.\frac{d}{dt} \ell(Y,  Q_{n}(t) (A,W))\right|_{t=t_{n}}
  \PHT D_{2}^{b} (P_{n}^{*}) = 0
\end{equation*}
(whereas it  is very  unlikely that $\PHT  D_{2}^{b} (P_{n}^{0})$ be  equal to
zero).  Consequently, (\ref{eq:targeted}) is met because
\begin{equation*}
  \PHT D^{b} (P_{n}^{*}) = 0.
\end{equation*}

Theorem~\ref{thm:main}    is   tailored    to   the    present    setting   in
Section~\ref{subsec:tailoring}.

\subsection{Variable importance measure of a continuous exposure}
\label{subsec:continuous}

In this section, $\xA \subset \xR$  is a bounded subset of $\xR$ containing 0,
which serves  as a  reference value.  Moreover,  we assume that  $P_{0,A|W} (A
\neq 0|W)  > 0$  $P_{0,W}$-almost surely  and the existence  of a  constant $c
(P_{0})>0$  such  that  $P_{0,A|W}  (A=0|W)  \geq  c(P_{0})$  $P_{0,W}$-almost
surely.  Introduced  in~\citep{NPVI,NPVI.app}, the true  parameter of interest
is
\begin{eqnarray}
  \notag
  \psi_{0}^{c}    &    \equiv     &    \mathop{\arg\min}_{\beta    \in    \xR}
  \espi{P_{0}}{\left(Y - \espi{P_{0}}{Y|A=0, W} - \beta A\right)^{2}} \\
  \label{eq:psi0:cont}
  & = & \mathop{\arg\min}_{\beta    \in    \xR}
  \espi{P_{0}}{\left(\espi{P_{0}}{Y|A,W}  -  \espi{P_{0}}{Y|A=0,  W}  -  \beta
      A\right)^{2}}
\end{eqnarray}
(the superscript ``$c$'' stands for ``continuous'').  

Let $\xM$  be the set  of finite  measures $P$ on  $\xO \equiv \xW  \times \xA
\times [0,1]$ equipped with the  Borel $\sigma$-field such that there exists a
constant $c(P)>0$  guaranteeing that  the marginal measure  of $\{w \in  \xW :
P_{A|W=w} (\xA \setminus \{0\}) > 0 \text{ and } P_{A|W=w}(\{0\}) \geq c(P)\}$
under $P_{W}$  equals $P(\xO)$.  In particular,  $P_{0} \in \xM$  by the above
assumption.

We see  $\psi_{0}^{c}$ as  the value at  $P_{0}$ of the  functional $\Psi^{c}$
characterized over $\xM$ by
\begin{equation}
  \label{eq:Psi:cont}
  \Psi^{c}      (P)     \equiv      \mathop{\arg\min}_{\beta      \in     \xR}
  \int_{\xA \times \xW} \left(Q_{P} (a,w) - Q_{P} (0,w) - \beta
    a\right)^{2} dP_{A|W=w}(a) dP_{W}(w),
\end{equation}
using  the  notation  of  Section~\ref{subsec:binary}.   By  Proposition~1  in
\citep{NPVI}, for each $P \in \xM$, 
\begin{equation*}
  \Psi^{c}   (P)    =   \frac{\int_{\xA   \times   \xW}    a(Q_{P}   (a,w)   -
    Q_{P}(0,w))dP_{A|W=w}(a)    dP_{W}(w)}{\int_{\xA    \times   \xW}    a^{2}
    dP_{A|W=w}(a) dP_{W}(w)}. 
\end{equation*}
If $P$ is a {\em distribution}, then
\begin{equation*}
  \Psi^{c}       (P)        =       \frac{\espi{P}{A(Q_{P}       (A,W)       -
      Q_{P}(0,W))}}{\espi{P}{A^{2}}}.
\end{equation*}

For  clarity,  we  introduce  some  notation.    For  each  $P  \in  \xM$  and
$(w,a) \in \xW \times \xA$, $\mu_{P}  (w) \equiv \int_{\xA} a dP_{A|W=w} (a)$,
and        $g_{P}        (0|w)        \equiv        P_{A|W=w}        (\{0\})$,
$\zeta^{2} (P) \equiv \int_{\xA} a^{2} dP_{A|W=w}  (a)$.  If $P (\xO) =1$, then
$\mu_{P}    (W)   =    \espi{P}{A|W}$,    $g_{P}(0|W)    =   P(A=0|W)$,    and
$\zeta^{2} (P) = \espi{P}{A^{2}}$.

\subsubsection*{Pathwise differentiability.}

A                     result                    similar                     to
Lemma~\ref{lem:pathwise:bin}~\cite[see][Proposition~1]{NPVI}  guarantees  that
$\Psi^{c}$ is  pathwise differentiable  like $\Psi^{b}$ with  influence curves
$D^{c}(P)\equiv D_{1}^{c} (P) + D_{2}^{c} (P) \in L_{0}^{2} (P)$,
\begin{eqnarray*}
  \zeta^{2} (P) D_{1}^{c}(P) (O) & \equiv & A\left(Q_{P}(A,W) - Q_{P}(0,W)
                                            - A\Psi^{c}(P)\right),\\ 
  \zeta^{2} (P) D_{2}^{c}(P) (O) & \equiv & (Y - Q_{P}(A,W)) \left(A -
                                            \frac{\mu_{P}(W) \1\{A=0\}}{g_{P}(0|W)}\right)
\end{eqnarray*}
(all $P \in \xM$).  Let $\calR^{c} : \xM^{2} \to \xR$ be characterized by
\begin{equation*}
  \calR^{c} (P, P') \equiv \Psi^{c} (P') - \Psi^{c} (P) - (P' - P)D^{c} (P).
\end{equation*}
as  in  (\ref{eq:second:order})  and  (\ref{eq:second:order:bin}). As  in  the
previous example,  $\calR^{c}$ satisfies (\ref{eq:linear:exp})  and, for every
$P,P' \in \xM$,
\begin{multline}
  \label{eq:eq:double:rob:cont:2}
  \calR^{c} (P,  P') =  \left(1 - \frac{\zeta^{2}  (P')}{\zeta^{2} (P)}\right)
  \left(\Psi^{c}(P')  -   \Psi^{c}(P)\right)  \\+   \frac{1}{\zeta^{2}(P)}  P'
  \left((Q_{P'}(0,\cdot)   -   Q_{P}(0,\cdot))    \left(\mu_{P'}   -   \mu_{P}
      \frac{g_{P'}(0|\cdot)}{g_{P}(0|\cdot)} \right)\right).
\end{multline}

Introduce 
\begin{equation*}
  \gamma_{n} \equiv  1 - \frac{\zeta^{2}(P_{0})}{\zeta^{2}(P_{n}^{*})} \quad
  \text{and}         \quad        \Gamma_{n}         \equiv         1        -
  \frac{\zeta_{n}^{2}(P_{0})}{\zeta_{n}^{2}(P_{n}^{*})} 
\end{equation*}
where    $\zeta_{n}^{2}(P_{0})$   and    $\zeta_{n}^{2}(P_{n}^{*})$   estimate
$\zeta^{2}(P_{0})$   and   $\zeta^{2}(P_{n}^{*})$.     With   these   choices,
(\ref{eq:eq:double:rob:cont:2}) guarantees  that {\bf A4} is  fulfilled in the
context of  this example when $\zeta^{2}(P_{n}^{*})$  converges in probability
to a finite real number such that $\gamma_{1} \neq 1$ and
\begin{equation*}
  \frac{1}{\zeta^{2}   (P_{n}^{*})}   P_{0}   \left((Q_{P_{0}}  (0,\cdot)   -
    Q_{P_{n}^{*}}     (0,\cdot))    \left(\mu_{P_{0}}     -    \mu_{P_{n}^{*}}
      \frac{g_{P_{0}}(0|\cdot)}{g_{P_{n}^{*}}(0|\cdot)} \right)\right) = o_{P}
  (1/\sqrt{n}).  
\end{equation*}
Through the product, we will  draw advantage of the synergistic convergences of
$Q_{P_{n}^{*}}  (0,\cdot)$  to  $Q_{P_{0}} (0,\cdot)$  and  $(\mu_{P_{n}^{*}},
g_{P_{n}^{*}})$   to  $(\mu_{P_{0}},   g_{P_{0}})$   (by  the   Cauchy-Schwarz
inequality  for example).   

\subsubsection*{Construction of the targeted estimator.}

Let  $\xQ^{w}$, $\xMu^{w}$  and $\xG^{w}$  be three  user-supplied  classes of
functions mapping $\xA \times \xW$,  $\xW$ and $\xW$ to $[0,1]$, respectively.
We first  estimate $Q_{P_0}$, $\mu_{P_{0}}$  and $g_{P_{0}}$ with  $Q_{n}$ and
$\mu_{n}$ and $g_{n}$ built  upon $\PHT$, $\xQ^{w}$, $\xMu^{w}$ and $\xG^{w}$.
For instance, one could simply minimize (weighted) empirical risks and define
\begin{eqnarray*}
  Q_{n} & \equiv & \mathop{\argmin}_{Q \in \xQ^{w}} \PHT \ell(Y, Q(A,W))
  =   \mathop{\argmin}_{Q \in  \xQ^{w}}  \sum_{i=1}^{N} \frac{\eta_{i}}{p_{i}}
  \ell(Y_{i}, Q(A_{i},W_{i})),\\
  \mu_{n} & \equiv & \mathop{\argmin}_{\mu \in \xMu^{w}} \PHT \ell(A, \mu(W))
  =  \mathop{\argmin}_{\mu \in  \xMu^{w}}  \sum_{i=1}^{N} \frac{\eta_{i}}{p_{i}}
  \ell(A_{i}, \mu(W_{i})), \\
  g_{n} & \equiv & \mathop{\argmin}_{g \in \xG^{w}} \PHT \ell(\1\{A=0\}, g(0|W))
  =  \mathop{\argmin}_{g \in  \xG^{w}}  \sum_{i=1}^{N} \frac{\eta_{i}}{p_{i}}
  \ell(\1\{A_{i}=0\}, g(0|W_{i}))
\end{eqnarray*}
(assuming  that  the  $\argmin$s  exist).   Alternatively,  one  could  prefer
minimizing cross-validated  (weighted) empirical risks.  We  also estimate the
marginal distribution $P_{0,W}$ of $W$ under $P_{0}$ with
\begin{equation}
  \label{eq:PHT:W1}
  \PHTW \equiv \frac{1}{N}\sum_{i=1}^N \frac{\eta_i}{p_i}\Dirac(W_i),
\end{equation}
and  the real-valued  parameter $\zeta^{2}  (P_{0})$ with  $\zeta^{2} (\PHTX)$
where  $\PHTX$  is  defined  as  in (\ref{eq:PHT:W1})  with  $X$  and  $X_{i}$
substituted for $W$ and $W_{i}$.

Let   $P_{n}^{0}$  be   a   measure  such   that   $Q_{P_{n}^{0}}  =   Q_{n}$,
$\mu_{P_{n}^{0}}      =      \mu_{n}$,      $g_{P_{n}^{0}}      =      g_{n}$,
$\zeta^{2} (P_{n}^{0}) =  \zeta^{2} (\PHTX)$, $P_{n,W}^{0} =  \PHTW$, {\em and
  from  which we  can sample  $A$ conditionally  on $W$}.   Picking up  such a
$P_{n}^{0}$  is  an easy  technical  task,  see~\citep[][Lemma~5]{NPVI} for  a
computationally    efficient    choice.     Then   the    initial    estimator
$\Psi^{b} (P_{n}^{0})$ of $\psi_{0}^{b}$ can be computed with high accuracy by
Monte-Carlo.  It  suffices to sample  a large  number $B$ (say  $B=10^{7}$) of
independent  $(A^{(b)},  W^{(b)})$  by   {\em  (i)}  sampling  $W^{(b)}$  from
$P_{n,W}^{0} = \PHTW$ then {\em  (ii)} sampling $A^{(b)}$ from the conditional
distribution  of  $A$  given  $W=W^{(b)}$  under  $P_{n}^{0}$  repeatedly  for
$b=1, \ldots, B$ and to make the approximation
\begin{equation}
  \label{eq:Psin0}
  \Psi^{c}   (P_{n}^{0})  \approx  \frac{B^{-1}   \sum_{b=1}^{B}  A^{(b)}
    (Q_{n}(A^{(b)}, W^{(b)}) - Q_{n}(0, W^{(b)}))}{\zeta^{2} (P_{n}^{0})}.
\end{equation}
However, the  construction $\Psi^{c} (P_{n}^{0})$ is  not tailored/targeted to
$\psi_{0}^{c}$ yet.  It is now time to target the inference procedure.

Targeting the inference procedure consists  in modifying $P_{n}^{0}$ in such a
way that the resulting  $P_{n}^{*}$ satisfies (\ref{eq:targeted}) with $D^{c}$
substituted for $D$. We proceed iteratively. Suppose that $P_{n}^{k}$ has been
constructed  for  some  $k  \geq   0$.   We  fluctuate  $P_{n}^{k}$  with  the
one-dimensional  parametric model  $\{P_{n}^{k} (t)  : t  \in \xR,  t^{2} \leq
c(P_{n}^{k})/\|D^{c} (P_{n}^{k})\|_{\infty}\}$ characterized by
\begin{equation*}
  \frac{dP_{n}^{k} (t)}{dP_{n}^{k}} = 1 + t D^{c} (P_{n}^{k}). 
\end{equation*}
Lemma~1  in~\citep{NPVI} shows  how $Q_{P_{n}^{k}(t)}$,  $\mu_{P_{n}^{k}(t)}$,
$g_{P_{n}^{k}(t)}$,  $\zeta^{2}  (P_{n}^{k}(t))$ and  $P_{n,W}^{k}(t)$  depart
from their  counterparts at $t=0$. The  optimal move along the  fluctuation is
indexed by
\begin{equation*}
  t_{n}^{k}   \equiv  \mathop{\arg\max}_{t}   \PHT  \log\left(1   +   t  D^{c}
    (P_{n}^{k})   \right),
\end{equation*}
{\em i.e.}, the maximum likelihood estimator of $t$ (note that the
random function $t  \mapsto \PHT \log(1 + t  D^{c} (P_{n}^{k}))$ is
strictly  concave).  It  results  in  the $(k+1)$-th  update  of  $P_{n}^{0}$,
$P_{n}^{k+1} \equiv P_{n}^{k} (t_{n}^{k})$.  

Contrary    to     what    happened     in    the    first     example,    see
Section~\ref{subsec:binary}, there  is no guarantee that  a $P_{n}^{k+1}$ will
coincide  with  its predecessor  $P_{n}^{k}$.   In  this  light, the  updating
procedure in  Section~\ref{subsec:binary} converged in one  single step. Here,
we  assume that the  iterative updating  procedure converges  (in $k$)  in the
sense  that, for  $k_{n}$  large enough,  $\PHT  D^{c}(P_{n}^{k_{n}}) =  o_{P}
(1/\sqrt{n})$.  We  set  $P_{n}^{*}  \equiv  P_{n}^{k_{n}}$.  It  is  actually
possible  to  come up  with  a one-step  updating  procedure  ({\em i.e.},  an
updating procedure such  that $P_{n}^{k} = P_{n}^{k+1}$ for  all $k\geq 1$) in
this  example  too  by   relying  on  so-called  universally  least  favorable
models~\citep{ULFM16}.   We  adopt  this  multi-step  updating  procedure  for
simplicity.

We can assume without loss of  generality that we can sample $A$ conditionally
on $W$  from $P_{n}^{*}$. The final  estimator is computed  with high accuracy
like $\Psi^{c} (P_{n}^{0})$ previously: with $Q_{n}^{*} \equiv Q_{P_{n}^{*}}$,
we sample $B$ independent $(A^{(b)}, W^{(b)})$ by {\em (i)} sampling $W^{(b)}$
from  $P_{n,W}^{*}$ then {\em  (ii)} sampling  $A^{(b)}$ from  the conditional
distribution of  $A$ given $W=W^{(b)}$ under $P_{n}^{*}$  repeatedly for $b=1,
\ldots, B$ and make the approximation
\begin{equation}
  \label{eq:Psin*}
  \psi_{n}^{*} \equiv \Psi^{c} (P_{n}^{*}) \approx \frac{B^{-1} \sum_{b=1}^{B}
    A^{(b)} (Q_{n}^{*}(A^{(b)},  W^{(b)}) - Q_{n}^{*}(0, W^{(b)}))}{\zeta^{2}
    (P_{n}^{*})}. 
\end{equation}

Theorem~\ref{thm:main}    is   tailored    to   the    present    setting   in
Section~\ref{subsec:tailoring}.

\subsection{Tailoring    the    main    theorem    in   the    settings    of
  Sections~\ref{subsec:binary} and \ref{subsec:continuous}}
\label{subsec:tailoring}

Consider  the following  assumptions for  the study  of $\psi_{n}^{*}$  in the
setting of Section~\ref{subsec:binary}:
\begin{description}
\item[A1$^{b}$] The classes $\xQ^w$  and $\xG^w$ are separable, $P_0 Q^2h^{-1}
  < \infty$  and $P_0 g^2h^{-1} < \infty$  for all $(Q, g)  \in \xQ^{w} \times
  \xG^{w}$, and  $J(1, \xQ^{w}, \|\cdot\|_{2,P_{0}})<  \infty$, $J(1, \xG^{w},
  \|\cdot\|_{2,P_{0}}) <  \infty$. Moreover, {\bf A2$^*$} is  met by $\xQ^{w}$
  and $\xG^{w}$.
\item[A2$^{b}$]    There     exists    $P_{1}     \in    \xM$     such    that
  $\|D^{b}(P_{n}^{*})  -  D^{b}  (P_{1})\|_{2,P_{0}} =  o_{P}(1)$.   Moreover,
  $\|Q_{n}^{*} - Q_{P_{0}}\|_{2,P_{0}}\times \|g_{n} - g_{P_{0}}\|_{2,P_{0}} =
  o_{P} (1/\sqrt{n})$ and  one knows a conservative  estimator $\Sigma_{n}$ of
  $P_{0} D^{b} (P_{1})^{2} h^{-1}$.
\end{description}

The assumptions  required for  the study of  $\psi_{n}^{*}$ in the  setting of
Section~\ref{subsec:continuous} are very similar:
\begin{description}
\item[A1$^c$] There exists a set $\xF \subset  \{D(P) : P \in \xM\}$ such that
  {\bf A1} and {\bf A2} are verified.
\item[A2$^{c}$]  There exist  $\zeta_{-}^{2} >  0$  and $P_{1}  \in \xM$  with
  $\zeta^{2}(P_1) \geq \zeta_{-}^{2} > 0$ such that
  \begin{gather*}
    \zeta^{2}(P_{n}^{*}) = \zeta^{2}(P_1) + O_{P} (1/\sqrt{n}),\\
    \|D^{c}(P_{n}^{*}) - D^{c} (P_{1})\|_{2,P_{0}} =o_{P}(1),\\
    \|Q_{n}^{*}   -   Q_{P_{0}}\|_{2,P_{0}}   \times   \left(\|\mu_{n}^{*}   -
      \mu_{P_{0}}\|_{2,P_{0}}  +  \|g_{n}  - g_{P_{0}}\|_{2,P_{0}}  \right)  =
    o_{P} (1/\sqrt{n}).
  \end{gather*}
  Moreover, $\Gamma_{n} - \gamma_{n} =  o_{P}(1)$ and one knows a conservative
  estimator $\Sigma_{n}$ of $P_{0} D^{b} (P_{1})^{2} h^{-1}$.
\end{description}

In {\bf A2$^{b}$}, $Q_{P_1}$ and $g_{P_1}$  should be interpreted as the limits of
$Q_{P_{n}^{*}}$ and $g_{P_{n}^{*}}$.  Likewise, $Q_{P_1}$, $\mu_{P_1}$ and $g_{P_1}$
in  {\bf A2$^{c}$}  should be  interpreted as  the limits  of $Q_{P_{n}^{*}}$,
$\mu_{P_{n}^{*}}$ and $g_{P_{n}^{*}}$.

\begin{corollary}
  \label{co:resultexamples}
  Set $\alpha  \in (0,1)$.  In the setting  of Section~\ref{subsec:binary} and
  under {\bf A1$^{b}$}, {\bf A2$^{b}$},
  \begin{equation*}
    \left[\psi_{n}^{*}           \pm                    
      \frac{\xi_{1-\alpha/2} \sqrt{\Sigma_{n}}}{\sqrt{n}} \right] 
  \end{equation*}
  is a confidence interval for $\psi_{0}^{b}$ with asymptotic coverage no less
  than  $(1-\alpha)$.  In the  setting of  Section~\ref{subsec:continuous} and
  under {\bf A1$^{c}$}, {\bf A2$^{c}$},
  \begin{equation*}
    \left[\psi_{n}^{*}           \pm                    
      \frac{\xi_{1-\alpha/2} \sqrt{\Sigma_{n}}}{(1-\Gamma_{n})\sqrt{n}} \right] 
  \end{equation*}
  is a confidence interval for $\psi_{0}^{c}$ with asymptotic coverage no less
  than $(1-\alpha)$.
\end{corollary}

\section{Simulation study}
\label{sec:sim}

We illustrate  the methodology with  the inference of the  variable importance
measure of a continuous exposure presented in Section~\ref{subsec:continuous}.
We  consider  three  data-generating  distributions $P_{0,1}$,  $P_{0,2}$  and
$P_{0,3}$ of a  data-structure $O = (W,A,Y)$.  The  three distributions differ
only in  terms of  the conditional variance  of $Y$  given $(A,W)$, but  do so
drastically. Specifically,  $O = (W,A,Y)$ drawn from  $P_{0,j}$ ($j=1,2,3$) is
such that
\begin{itemize}
\item $W \equiv  (V, W_{1}, W_{2})$ with $P_{0} (V=1) =  1/6$, $P(V=2) = 1/3$,
  $P(V=3)=1/2$  and, conditionally  on  $V$, $(W_{1},  W_{2})$  is a  Gaussian
  random      vector       with      mean      $(0,0)$       and      variance
  $\begin{smatrix}[c]1&-0.2\\-0.2&1\end{smatrix}$  (if  $V=1$), $(1,1/2)$  and
  $\begin{smatrix}[c]0.5&0.1\\0.1&0.5\end{smatrix}$ (if $V=2$), $(1/2, 1)$ and
  $\begin{smatrix}1&0\\0&1\end{smatrix}$ (if $V=3$);
\item conditionally  on $W$, $A=0$ with  probability 80\% if  $W_{1} \geq 1.1$
  and $W_{2} \geq 0.8$ and  10\% otherwise; moreover, conditionally on $W$ and
  $A \neq 0$, $A-1$ is drawn from the $\chi^{2}$-distribution with 1 degree of
  freedom  and non-centrality parameter  $\sqrt{(W_{1} -  1.1)^{2} +  (W_{2} -
    0.8)^{2}}$;
\item conditionally  on $(W,A)$, $Y$ is  a Gaussian random  variable with mean
  $\espi{P_{0}} {Y|A,W}  \equiv A(W_1  + W_2)/6  + W_1 +  W_2/4 +  \exp((W_1 +
  W_2)/10)$ and standard deviation
  \begin{itemize}
  \item[-] 1.5 (if $V=1$), 1 (if $V=2$) and 0.5 (if $V=3$) for $j=1$;
  \item[-] 1 (if $V=1$), 5 (if $V=2$) and 10 (if $V=3$) for $j=2$;
  \item[-] 50 (if $V=1$), 10 (if $V=2$) and 1 (if $V=3$) for $j=3$.
  \end{itemize}
\end{itemize}
The   unique   true  parameter   is   $\psi_{0}^{c}   =  \Psi^{c}(P_{0,1})   =
\Psi^{c}(P_{0,2}) = \Psi^{c}(P_{0,3})$. It equals approximately 0.1204. 

For  $B=10^{3}$ and  each  $j=1,2,3$, we  repeat  independently the  following
steps:
\begin{enumerate}
\item simulate  a data set  of $N=10^{7}$ independent observations  drawn from
  $P_{0,j}$;
\item\label{item:init}  extract $n_{0}  \equiv 10^{3}$  observations  from the
  data  set  by survey  sampling  with $h_{0}\equiv  1$,  and  based on  these
  observations:
  \begin{enumerate}
  \item apply  the procedure described  in Section~\ref{subsec:continuous} and
    retrieve $D^{c} (P_{n_{0}}^{k_{n_{0}}})$;
  \item  set $f_{n_{0},1}  \equiv D^{c}  (P_{n_{0}}^{k_{n_{0}}})$  and regress
    $f_{n_{0},1} (O)^{2}$  on $V$, call  $f_{n_{0},2}$ the square root  of the
    resulting conditional expectation, see (\ref{eq:f2});
  \item   estimate  the   marginal  distribution   of  $V$,   estimate  $P_{0}
    f_{n_{0},2}$     with     $\pi_{n_{0},2}$     and    set     $h     \equiv
    f_{n_{0},2}/\pi_{n_{0},2}$;
  \end{enumerate}
\item  for each  $n$ in  $\{10^{3},  5\times 10^{3},  10^{4}, 5\times  10^{4},
  10^{5}\}$, successively, extract by survey sampling with $h$ a sub-sample of
  $n$ observations from  the data set (deprived of  the observations extracted
  in  step~\ref{item:init})  and,  based  on  these  observations,  apply  the
  procedure described in  Section~\ref{subsec:continuous}. We use $\Sigma_{n}$
  given in  (\ref{eq:Sigma:n}) to  estimate $\sigma_{1}^{2}$, although  we are
  not sure in advance that it is a conservative estimator.
\end{enumerate}
We thus obtain  $15\times B$ estimates of $\psi_{0}^{c}$  and their respective
confidence intervals.

To give an idea of what is the optimal $h$ in each case, we save the result of
step~2 in the above list in the  first of the $B$ simulations under $P_{0,1}$,
$P_{0,2}$ and $P_{0,3}$. So, the optimal $h$ equals approximately
\begin{itemize}
\item[-] $h_{1}$  given by $(h_{1} (1),  h_{1} (2), h_{1}  (3)) \approx (1.03,
  0.67, 1.21)$ under $P_{0,1}$;
\item[-] $h_{2}$  given by $(h_{2} (1),  h_{2} (2), h_{2}  (3)) \approx (0.30,
  0.60, 1.50)$ under $P_{0,2}$;
\item[-] $h_{3}$  given by $(h_{3} (1),  h_{3} (2), h_{3}  (3)) \approx (4.66,
  0.53, 0.09)$ under $P_{0,3}$
\end{itemize}
Note  how  different are  $h_{1}$,  $h_{2}$  and  $h_{3}$ (to  facilitate  the
comparisons, $h_{1}$, $h_{2}$ and $h_{3}$ are renormalized to satisfy $P_{0,j}
h_{j} = 1$ for $j=1,2,3$).

Applying  the  TMLE procedure  is  straightforward  thanks  to the  \texttt{R}
package  called  \texttt{tmle.npvi}~\cite{tmle.npvi,NPVI.app}. Note,  however,
that it  is necessary to compute $\Gamma_{n}$  and $\Sigma_{n}$. Specifically,
we fine-tune the  TMLE procedure by setting \texttt{iter}  (the maximum number
of iterations  of the  targeting step) to  7 and  \texttt{stoppingCriteria} to
\texttt{list(mic=0.01,  div=0.01, psi=0.05)}.   Moreover, we  use  the default
\texttt{flavor}  called \texttt{"learning"}, thus  notably rely  on parametric
linear  models  for  the  estimation of  the  infinite-dimensional  parameters
$Q_{P_{0}}$, $\mu_{P_{0}}$  and $g_{P_{0}}$  and their fluctuation.   We refer
the interested reader to the package's manual and vignette for details.

Sampford's  sampling method~\cite{sampford1967sampling} implements  the survey
sampling  described in  Section~\ref{sec:surv:samp}. However,  when  the ratio
$n/N$ is close to 0  or 1, this acceptance-rejection algorithm typically takes
too much time to succeed.  In our setting, this is the case when $n/N$ differs
from $10^{-3}$. To  circumvent that issue, we approximate  the survey sampling
described  in  Section~\ref{sec:surv:samp}  with a  Pareto  sampling~\cite[see
Algorithm~2 in][Section~5]{Bondesson06}. 

\begin{table}
  \centering
  \begin{tabular}{r|ccccc|ccccc}
    \multicolumn{1}{c}{}&  \multicolumn{5}{|c|}{$P_{0,1}$, optimal  $h_{1}$} &
                                                                               \multicolumn{5}{|c}{$P_{0,1}$, $h_{0}\equiv 1$} \\ \hline 
    $n$ &  b. & $p$-val.  &  c. &  v.  & e.  v. &  b. & $p$-val.  &  c. &
                                                                          v. & e. v. \\ 
    \hline 
    $1 \times 10^3$ & 0.024 & 0.018 & 0.957 & 0.946 & 1.149 & 0.025& 0.499  & 0.963 & 1.010 & 1.219 \\
    $5 \times 10^3$ & 0.011 & 0.858 & 0.971 & 0.972 & 1.199 & 0.011 & 0.320 & 0.968 & 0.981 & 1.265 \\
    $1\times 10^4$ & 0.008 & 0.948 & 0.970 & 0.961 & 1.210 & 0.008 & 0.215 & 0.964 & 1.060 & 1.277 \\
    $5 \times 10^4$ & 0.004 & 0.441 & 0.920 & 1.334 & 1.213 & 0.004 & 0.253 & 0.916 & 1.282 & 1.283 \\
    $1\times 10^5$ & 0.004 & 0.858 & 0.861 & 1.601 & 1.214 & 0.004 & 0.750 & 0.874 & 1.664 & 1.284 \\
    \multicolumn{7}{c}{}\\
                        & \multicolumn{5}{|c|}{$P_{0,2}$, optimal $h_{2}$} & \multicolumn{5}{|c}{$P_{0,2}$,
                                                                             $h_{0} \equiv 1$} \\\hline 
    $n$ &  b. & $p$-val. &   c. &  v.  & e.  v. &  b.  & $p$-val.&   c. &
                                                                          v. & e. v.\\\hline 
    $1 \times 10^3$ & 0.110 & 0.001 & 0.955 & 20.14 & 26.01 & 0.124 & 0.001 & 0.945 & 25.49 & 30.32 \\ 
    $5 \times 10^3$ & 0.045 & 0.526 & 0.986 & 16.08 & 25.84 & 0.052 & 0.156 & 0.978 & 21.46 & 32.00 \\ 
    $1\times 10^4$ & 0.032 & 0.419 & 0.991 & 16.34 & 25.83 & 0.036 & 0.686 & 0.983 & 20.69 & 32.17 \\ 
    $5 \times 10^4$ & 0.015 & 0.501 & 0.990 & 16.69 & 25.89 & 0.016 & 0.775 & 0.989 & 20.01 & 32.45 \\ 
    $1\times 10^5$ & 0.011 & 0.956 & 0.985 & 17.20 & 25.88 & 0.012 & 0.839 & 0.986 & 20.23 & 32.38 \\
    \multicolumn{7}{c}{}\\
                        & \multicolumn{5}{|c|}{$P_{0,3}$, optimal $h_{3}$} & \multicolumn{5}{|c}{$P_{0,3}$,
                                                                             $h_{0} \equiv 1$} \\\hline 
    $n$ &  b. & $p$-val.  &  c. &  v. & e.  v. &  b. & $p$-val. &  c. &
                                                                        v. & e. v.\\\hline 
    $1 \times 10^3$ & 0.229 & 0.001 & 0.987 & 86.85 & 184.2 & 0.532 & 0.001 & 0.910 & 518.5 & 549.6\\
    $5 \times 10^3$ & 0.093 & 0.242 & 0.994 & 70.15 & 175.7 & 0.181 & 0.001 & 0.994 & 264.6 & 627.7\\
    $1\times 10^4$ & 0.069 & 0.268 & 0.997 & 73.32 & 174.5 & 0.127 & 0.022 & 0.995 & 253.8 & 629.4 \\
    $5 \times 10^4$ & 0.029 & 0.085 & 1.000 & 65.54 & 174.0 & 0.055 & 0.459 & 0.999 & 228.3 & 642.7 \\
    $1\times 10^5$ & 0.022 & 0.584 & 0.998 & 73.98 & 174.2 & 0.040 & 0.054 & 1.000 & 242.7 & 644.5 \\
  \end{tabular}
  \caption{Summarizing the  results of the  simulation study. The  top, middle
    and bottom tables correspond to simulations under $P_{0,1}$, $P_{0,2}$ and
    $P_{0,3}$.  Each of them reports the empirical bias of the estimators (b.,
    $B^{-1} \sum_{b=1}^{B}  |\psi_{n,b}^{*} - \psi_{0}^{c}|$), $p$-value  of a
    Shapiro-Wilk  test  of normality  ($p$-val.),  empirical  coverage of  the
    confidence                          intervals                         (c.,
    $B^{-1}  \sum_{b=1}^{B} \1\{\psi_{0}^{c}  \in I_{n,b}\}$),  $n$ times  the
    empirical        variance       of        the       estimators        (v.,
    $n[B^{-1}       \sum_{b=1}^{B}       \psi_{n,b}^{*2}       -       (B^{-1}
    \sum_{b=1}^{B}\psi_{n,b}^{*})^{2}]$) and  empirical mean of $n$  times the
    estimated      variance      of      the     estimators      (e.       v.,
    $B^{-1} \sum_{b=1}^{B}  \Sigma_{n,b}$), for every sub-sample  size $n$ and
    for both $h$ optimal and $h = h_{0} \equiv 1$.}
  \label{tab:sim}
\end{table}

The  results are  summarized in  Table~\ref{tab:sim}.  We  first focus  on the
empirical  bias  of the  TMLE  and  $p$-values  of  the Shapiro-Wilk  test  of
normality of its distribution.  In  all settings, the empirical bias decreases
as $n$ grows (under $P_{0,1}$, the empirical biases for $n=5\times 10^{4}$ and
$n=10^{5}$ equal 0.0044 and 0.0036 when relying on $h_{1}$ or $h_{0}$).  Under
each $P_{0,j}$  and for every sub-sample  size, the empirical bias  is smaller
when relying  on $h_{j}$  than on $h_{0}$,  approximately twice  smaller under
$P_{0,3}$.  As expected due to  our choices of conditional standard deviations
of $Y$ given $(A,W)$, the empirical  bias is larger under $P_{0,3}$ than under
$P_{0,2}$  and larger  under  $P_{0,2}$ than  under  $P_{0,1}$.  Except  under
$P_{0,3}$  when relying  on $h_{0}$,  for  every $n\geq  5\times 10^{3}$,  the
$p$-values  of  the Shapiro-Wilk  test  of  normality  are coherent  with  the
convergence in law of the TMLE to a Gaussian distribution. Under $P_{0,3}$ and
when relying on $h_{0}$, there is more evidence of a departure from a Gaussian
distribution. Inspecting the  results of the simulations  studies reveals that
this is mostly due to slightly too heavy tails.

We now  focus on the  empirical coverage, empirical  variance and mean  of the
estimated variance of the TMLE. Consider  the table about the simulation under
$P_{0,1}$  first.   For  $n  \in  \{10^{3},  5\times  10^{3},  10^{4}\}$,  the
empirical coverage is satisfying when relying on both $h_{1}$ and $h_{0}$.  At
each of these sub-sample sizes, it  does seem that we achieve the conservative
estimation of $\sigma_{1}^{2}$.  However,  the empirical coverage deteriorates
sharply  for  $n   \in  \{5  \times  10^{4},  10^{5}\}$.    It  appears  that,
concomitantly, the  empirical variance  of the estimators  increases strongly.
This may be due to the fact that, here, $n$ is not that small compared to $N$,
so that neglecting the second LHS  term in (\ref{eq:decomp}) is inadequate, so
that $\sigma_{1}^{2}$ is not the  limiting variance.  In conclusion, note that
resorting to the  optimal $h$ does not  yield much gain in  terms of empirical
variance of the estimators.

We now turn  to the two remaining  tables. The first striking  feature is that
the  empirical coverage  exceeds largely  the nominal  coverage of  95\%.  The
comparison of the  empirical variance with the mean of  the estimated variance
reveals    that   we    do    achieve   the    conservative   estimation    of
$\sigma_{1}^{2}$. The second  striking feature is that  the empirical variance
stabilizes  for  $n$  larger  than  $10^3$, contrary  to  what  happens  under
$P_{0,1}$.  It still holds that $n$ may  not be small compared to $N$. Perhaps
this  is  counterbalanced  by  the   fact  that,  by  increasing  starkly  the
conditional  variance  of $Y$  given  $(A,W)$  under $P_{0,2}$  and  $P_{0,3}$
relative to $P_{0,1}$, we make $P_{0} f_{1}^{2} h^{-1}$, the first LHS term in
(\ref{eq:decomp}),    much     larger    than    the    second     LHS    term
$n(P_{0} f_{1})^{2}/N$.   Finally, resorting to  the optimal $h$  yields, both
under  $P_{0,2}$  and $P_{0,3}$,  considerable  gains  in terms  of  empirical
variance  of  the estimators  and  in  terms of  the  width  of the  resulting
confidence intervals.

\subsubsection*{Acknowledgements.}  

The  authors acknowledge  the support  of the  French Agence  Nationale  de la
Recherche (ANR), under grant ANR-13-BS01-0005 (project SPADRO).

\appendix

\section{Proof of Theorem~\ref{thm:main}}
\label{sec:proofs}

Throughout the proofs, ``$a \lesssim  b$'' means that there exists a universal
constant $L>0$ such that $a\le Lb$.

We   start  with   a  central   limit  theorem   for  the   empirical  process
$(\sqrt{n}(\PHT-P_0)f)_{f\in\xF}$.   Its proof  is given  at the  end  of this
section.      Recall    that    a     random    process     $\mathbb{G}$    in
$\ell^{\infty}(\mathcal{F})$  is  $\|\cdot\|_{2,P_{0}}$-equicontinuous if  for
each   $\xi>0$,   there  exists   $\delta>0$   such   that,   for  all   $f,f'
\in\mathcal{F}$,             $\|f-f'\|_{2,P_{0}}\le\delta$             implies
$P_0(|\mathbb{G}(f-f')|)\leq \xi$.

\begin{theorem}
  \label{thm:result_sondage}
  Under     {\bf      A1}     and     {\bf     A2}      there     exists     a
  $\|\cdot\|_{2,P_{0}}$-equicontinuous             Gaussian            process
  $\mathbb{G}^{h}\in\ell^{\infty}(\mathcal{F})$   with   covariance   operator
  $\Sigma$  such that  $(\sqrt{n}(\PHT-P_0)f)_{f\in\xF}$  converges weakly  in
  $\ell^{\infty}(\mathcal{F)}$ towards $\lG$. The same result holds with $\xF$
  replaced by $\{f - f_{1} : f \in \xF\}$.
\end{theorem}

We  now  turn  to  the  proof of  Theorem  \ref{thm:main}.   Since  $P_{n}^{*}
D(P_{n}^{*})  = 0$ (by  definition, the  influence function  $D(P_{n}^{*})$ is
centered under $P_{n}^{*}$), {\bf A4} rewrites
\begin{equation*}
  \calR(P_{n}^{*},  P_{0}) = \psi_{0}  - \psi_{n}^{*}  - P_{0}  D(P_{n}^{*}) =
  -\gamma_{n} (\psi_{n}^{*} - \psi_{0}) + o_{P}(1/\sqrt{n}),
\end{equation*}
hence 
\begin{equation*}
  (1-\gamma_{n})  \sqrt{n}  (\psi_{n}^{*}  -  \psi_{0})  =  -  \sqrt{n}  P_{0}
  D(P_{n}^{*}) + o_{P}(1). 
\end{equation*}
Moreover, (\ref{eq:targeted}) implies that the above equality also yields
\begin{eqnarray*}
  \notag
  (1-\gamma_n)\sqrt{n}(\psi_{n}^*       -       \psi_0)       &      =       &
  \sqrt{n}(\PHT - P_{0}) D(P_n^*) + o_{P} (1)\\
  & =  & \sqrt{n}(\PHT  - P_{0})  f_{1} + \sqrt{n}(\PHT  - P_{0})  (D(P_n^*) -
  f_{1}) + o_{P} (1). 
\end{eqnarray*}

Theorem~\ref{thm:result_sondage} implies in  particular that $\sqrt{n} (\PHT -
P_{0})f_{1}$  converges in  law  to the  centered  Gaussian distribution  with
variance $\Sigma(f_{1}, f_{1})$. 

Let us prove now that $\sqrt{n}(\PHT  - P_{0}) (D(P_n^*) - f_{1}) = o_{P}(1)$.
This   is   a   consequence   of  Theorem~\ref{thm:result_sondage}   and   the
concentration inequality of~\cite[][Corollary~2.2.8]{van1996weak}.

Let $\|\cdot\|_{2,\Sigma}$ be the norm on $\xF$ given by $\|f\|_{2,\Sigma}^{2}
\equiv \Sigma(f,f)$. For every $\delta > 0$, introduce
\begin{equation*}
  \xF_\delta\equiv \{f\in\xF : P_0(f-f_1)^2\le \delta^2\} \subset \xF.
\end{equation*}
The  diameter   of  $\xF_{\delta}$  wrt  $\|\cdot\|_{2,\Sigma}$   is  at  most
$\delta/\sqrt{c(h)}$.  By  \cite[][Corollary~2.2.8]{van1996weak},
\begin{eqnarray}
  \notag
  \espi{0}{\sup_{f\in \xF_{\delta}}\lG (f-f_1)} & \lesssim &  \int_0^{\delta/\sqrt{c(h)}} \sqrt{\log N 
    (\epsilon,\xF_{\delta},\|\cdot\|_{2,\Sigma})}d\epsilon\\
  \label{eq:concentration}
  & \lesssim & \int_0^{\delta/\sqrt{c(h)}} \sqrt{\log N 
    (\epsilon,\xF,\|\cdot\|_{2,\Sigma})}d\epsilon.
\end{eqnarray}
Set arbitrarily $\alpha,\beta>0$, and choose $\delta>0$ in such a way that
\begin{equation*}
  \int_0^{\delta/\sqrt{c(h)}}      \sqrt{\log      N      (\epsilon,      \xF,
    \|\cdot\|_{2,\Sigma})}d\epsilon \le \alpha\beta. 
\end{equation*}
By Markov's  inequality, (\ref{eq:concentration})  and choice of  $\delta$, it
holds that
\begin{equation*}
  \probi{0}{\sup_{f\in \xF_{\delta}}\lG (f-f_1)\ge \alpha} \le \alpha^{-1} 
  \espi{0}{\sup_{f\in \xF_{\delta}}\lG (f-f_1)} \lesssim \beta.
\end{equation*}
Hence, Theorem~\ref{thm:result_sondage} implies that, for $n$ large enough,
\begin{equation}
  \label{eq:part:one}
  \probi{0}{\sqrt{n}(\PHT-P_0)(D(P_n^*))-f_1)\ge \alpha} \le 2 \beta.  
\end{equation}
Furthermore,  by {\bf  A3},  $P_{0} (D(P_{n}^{*})  \not\in \xF_{\delta})  \leq
\beta$   for    $n$   large    enough.    Combining   this    inequality   and
(\ref{eq:part:one}) finally yields
\begin{eqnarray*}
  \probi{0}{\sqrt{n}(\PHT-P_0)(D(P_n^*))-f_1)\ge         \alpha}         &\le&
  \probi{0}{\sqrt{n}(\PHT-P_0)(D(P_n^*))-f_1)\ge  \alpha\   ,\  D(P_n^*))  \in
    \xF_\delta}\\
  & & \quad +\probi{0}{D(P_n^*)) \notin \xF_\delta}\\
  &\le& \probi{0}{\sup_{f\in \xF_\delta}\sqrt{n}(\PHT-P_0)(f-f_1)\ge \alpha}\\
  && \quad +\probi{0}{D(P_n^*)) \notin \xF_\delta} \leq 3\beta
\end{eqnarray*}
for $n$ large enough. 

Consequently, $(1-\gamma_{n})  (\psi_{n}^{*} - \psi_{0})$ converges  in law to
the    centered   Gaussian    distribution   with    variance   $\Sigma(f_{1},
f_{1})$. Applying Slutsky's lemma completes the proof.

\subsection*{Proof of Theorem \ref{thm:result_sondage}.}


The  proof  relies on  results  from~\citep{hajek1964,barbour1990}.  For  each
$f\in \xF$, define
\[
Z_{N}(f) \equiv \PHT f
\]
and
\[
\pG (f) \equiv \sqrt{n}(\PHT-P_{0})f = \sqrt{n}(Z_{N}(f)-P_{0}f).
\]
We first state and prove  the following lemma, by using \citep[][Lemma~4.3 and
Theorem~7.1]{hajek1964}:

\begin{lemma}
  \label{lem:singlef}
  For every  (measurable) real-valued function  $f$ on $\xO$ such  that $P_{0}
  f^{2}/h$  is finite, $\pG  (f)$ converges  in law  to the  centered Gaussian
  distribution        with        variance        $\sigma^{2}(f)        \equiv
  \espi{P_{0}}{f^2(O)h(V)^{-1}}$.
\end{lemma}

\begin{proof}[Proof of Lemma~\ref{lem:singlef}] This is a three-step proof. 

  {\em Step~1: preliminary.}  Set arbitrarily a measurable function $f:\xO \to
  \xR$ such that $P_{0}f^{2}/h$ is finite and define
  \[
  T_{N}(f) \equiv \frac{1}{N}\sum_{i=1}^N \frac{f(O_i)\varepsilon_i}{p_i}.
  \]
  The only difference  between $T_{N} (f)$ and $Z_{N}(f)$  is the substitution
  of  $(\varepsilon_{1},  \ldots,  \varepsilon_{N})$ for  $(\eta_{1},  \ldots,
  \eta_{N})$.     Since    $(O_{1},    \varepsilon_{1}),    \ldots,    (O_{N},
  \varepsilon_{N})$ are independently sampled (from $P_{0}^{\varepsilon}$), it
  holds that $\espi{P_{0}^{\varepsilon}}{T_{N}(f)} = P_{0} f$ and
  \begin{eqnarray}
    \notag
    \vari{P_{0}^{\varepsilon}}{T_{N}(f)} & = & \frac{1}{N}
    \vari{P_{0}^{\varepsilon}}{\frac{f(O_{1})\varepsilon_{1}}{p_{1}}} \\
    \notag
    &   =  &   \frac{1}{n}  \espi{P_{0}}{\frac{f^2(O)}{h(V)}}   -  \frac{1}{N}
    \left(P_{0} f\right)^{2} \\ 
    \label{eq:arg:var}
    & = &     \frac{\sigma^2(f)}{n} + o\left(1/n\right).
  \end{eqnarray}
  Thus,  $\sqrt{n} (T_{N}(f)  - P_{0}  f)$ converges  in law  to  the centered
  Gaussian distribution with variance $\sigma^{2}(f)$. The challenge is now to
  derive another central limit theorem for $Z_{N}(f)$ from this convergence in
  law. 

  {\em Step~2:  coupling.}  The rest of  the proof mainly  hinges on coupling.
  We may assume without loss of  generality that there exist $U_1, \dots, U_N$
  independently drawn from the uniform distribution on $[0,1]$ and independent
  of  $(O_{1},  \ldots,  O_{N})$ such  that,  for  each  $1  \leq i  \leq  N$,
  $\varepsilon_{i} = \1\{U_{i} \leq  p_{i}\}$.  We now define $\ell_{N} \equiv
  n/\sum_{i=1}^N    p_i$   and,    for    each   $1    \leq    i   \leq    N$,
  $\varepsilon_i(\ell_{N}) =\1\{U_i  \le \ell_{N}  p_i\}$.  This is  the first
  coupling used in the proof.

  The second  coupling is  more elaborate.  Due  to~\citeauthor{hajek1964}, it
  gives rise to  two random subsets $s_{K}$ and $s_{n}$  of $\{1, \ldots, N\}$
  that we  characterize now, in three  successive steps.  In the  rest of this
  step of the proof, we work conditionally on $O_{1}, \ldots, O_{N}$.
  \begin{enumerate}
  \item Drawing $s_{n} \subset \{1, \ldots, N\}$: 
    \begin{enumerate}
    \item  sample  $(\eta_{1}',   \ldots,  \eta_{N}')$  from  the  conditional
      distribution  of  $(\varepsilon_{1}',  \ldots, \varepsilon_{N}')$  given
      $\sum_{i=1}^{N}  \varepsilon_{i}' =  n$ when  $\varepsilon_{1}', \ldots,
      \varepsilon_{N}'$   are   independently   drawn   from   the   Bernoulli
      distributions with parameters $\ell_{N}p_{1}, \ldots, \ell_{N}p_{N}$,
      respectively;
    \item define $s_{n} = \{1\leq i  \leq N : \eta_{i}'=1\}$ and $D_{n} \equiv
      \sum_{i \in s_{n}} (1 - \ell_{N} p_{i})$ for future use.
    \end{enumerate}
    We say simply that $s_{n}$ is  drawn from the rejective sampling scheme on
    $\{1, \ldots,  N\}$ with  parameter $(\ell_{N}p_{i} :  i \in  \{1, \ldots,
    N\})$ (see Section~\ref{sec:main}).
  \item Drawing $K \in \{1, \ldots, N\}$: 
    \begin{enumerate}
    \item sample  $\varepsilon_{1}'', \ldots, \varepsilon_{N}''$ independently
      from  the   Bernoulli  distributions  with   parameters  $\ell_{N}p_{1},
      \ldots,$ $\ell_{N}p_{N}$, respectively;
    \item define $K \equiv \sum_{i=1}^{N} \varepsilon_{i}''$.
    \end{enumerate}
  \item Drawing $s_{K} $:
    \begin{enumerate}
    \item if $K =n$, then set $s_{K} \equiv s_{n}$;
    \item if $K  > n$, then draw $s_{K-n}$ from  the rejective sampling scheme
      on    $\{1,    \ldots,    N\}    \setminus   s_{n}$    with    parameter
      $((K-n)\ell_{N}p_{i}/D_{n} :  i \in  \{1, \ldots, N\}  \setminus s_{n})$
      and set $s_{K} \equiv s_{n} \cup s_{K-n}$;
    \item if $K  < n$, then draw $s_{n-K}$ from  the rejective sampling scheme
      on $s_{n}$ with parameter $((K-n)\ell_{N}p_{i}/D_{n} : i \in s_{n})$ and
      set $s_{K} \equiv s_{n} \setminus s_{n-K}$.
    \end{enumerate}
  \end{enumerate}

  We denote  by $\xS$ the joint  law of $(s_K,s_n)$. Obviously,  $\xS$ is such
  that $s_K \subset  s_n$ or $s_n \subset s_K$  $\xS$-almost surely. We denote
  by $\xP$ the law of the Poisson sampling scheme, {\em i.e.}, the law of $\{1
  \leq i \leq N : \varepsilon_{i}'=1\}$ from the description of how $s_{n}$ is
  drawn.  Law  $\xS$ is  a coupling  of the rejective  sampling scheme  and an
  approximation  to the  Poisson sampling  scheme $\xP$  in the  sense  of the
  following corollary of~\cite[][Lemma~4.3]{hajek1964}.
  \begin{proposition}[Hajek]
    \label{pr:coupling}
    If $d_N  \equiv \sum_{i=1}^N p_i(1-p_i)$ goes  to infinity as  $N$ goes to
    infinity, then the marginal  distribution of $s_{K}$ when $(s_{K}, s_{n})$
    is drawn from $\xS$ converges to $\xP$ in total variation.
  \end{proposition}
  The condition on $d_N$ is met for our choice of $(p_1, \ldots, p_{N})$.

  {\em Step~3: concluding.} Introduce 
  \begin{eqnarray*}
    T_{N}^{\ell_{N}}(f)        &         \equiv        &        \frac{1}{N}\sum_{i=1}^N
    \frac{f(O_i)\varepsilon_i(\ell_{N})}{\ell_{N} p_i}, \\ 
    T_{N}^{s_{K}}(f) & \equiv & \sum_{i \in s_K} \frac{f(O_i)}{p_i},\\
    T_{N}^{s_{n}}(f) & \equiv  & \sum_{i \in s_n} \frac{f(O_i)}{p_i}.
  \end{eqnarray*}
  The random variables  $Z_{N} (f)$, $T_{N}^{\ell_{N}}(f)$, $T_{N}^{s_{K}}(f)$
  and $T_{N}^{s_{n}}(f)$ satisfy the following properties.
  \begin{itemize}
  \item $Z_{N} (f)$ and $T_{N}^{s_{n}} (f)$ share a common law.\\
    This is a straightforward consequence of Proposition~\ref{pr:coupling}.
  \item $\sqrt{n} (T_{N}^{s_{n}} (f) - T_{N}^{s_{K}} (f)) = o_{P} (1)$.\\
    Indeed, it is  shown in the proof of  \cite[][Theorem 7.1]{hajek1964} that
    the  convergence of  $d_N$ (defined  in  Proposition~\ref{pr:coupling}) to
    infinity  implies,  conditionally  on  $O_{1}, \ldots,  O_{N}$,  $\sqrt{n}
    (T_{N}^{s_{K}}(f)  -  T_{N}^{s_{n}}(f)) =  o_{P}  (1)$. The  unconditional
    result readily follows.
  \item $T_{N}^{s_{K}} (f)$ and $T_{N}^{\ell_{N}} (f)$ have asymptotically the
    same law, in the sense that the total variation distance between their
    laws goes to 0 as $N$ goes to infinity.\\
    This is a consequence of Proposition~\ref{pr:coupling}.
  \item $\sqrt{n} (T_{N}^{\ell_{N}} (f) - T_{N} (f)) = o_{P} (1)$.\\
    It   suffices   to    show   that   $\espi{P_{0}^{\varepsilon}}{(\T(f)   -
      T_{N}^{\ell_{N}}(f))^2} = o(1/n)$. Observe now that
    \begin{eqnarray*}
      n\espi{P_{0}^{\varepsilon}}{\left(T_{N}^{\ell_{N}}(f)-\T(f)\right)^2} & =
      & \frac{n}{N} 
      \espi{P_{0}^{\varepsilon}}{\left(\frac{\varepsilon_{1}(\ell_{N})}{\ell_{N}}
          - \varepsilon_{1}\right)^2 \frac{f(O_{1})^2}{p_{1}^2}}\\
      & = & 
      \espi{P_{0}^{\varepsilon}}{\left(\frac{1}{\ell_{N}}   -   2   \frac{\min
            (1,\ell_{N})}{\ell_{N}}+1\right) \frac{f(O_{1})^2}{h(V_{1})}}.  
    \end{eqnarray*}
    The  strong law of  large numbers  yields that  $\ell_{N}$ converges  to 1
    almost surely, which implies $1/\ell_{N} - 2 \min (1,\ell_{N})/\ell_{N}+1$
    converges to 0 almost surely hence the result by the dominated convergence
    theorem.
  \end{itemize}
  Consequently,  $\pG(f) \equiv  \sqrt{n} (Z_{N}(f)  - P_{0}f)$  and $\sqrt{n}
  (T_{N} (f) - P_{0} f)$ have  asymptotically the same law. The same arguments
  are valid when $\{f  - f_{1} : f \in \xF\}$ is  substituted for $\xF$. Thus,
  the proof is complete.
\end{proof}

We  can  now  prove  Theorem~\ref{thm:result_sondage}.   We  first  note  that
Lemma~\ref{lem:singlef}  implies the asymptotic  tightness of  the real-valued
random    variable    $\pG(f)$   for    all    $f    \in   \xF$.     Moreover,
Lemma~\ref{lem:singlef} and the Cram\'er-Wold  device yield the convergence in
law  of  $(\pG  f_1,\dots,\pG  f_M)$  to $(\lG  f_1,\dots,\lG  f_M)$  for  all
$(f_1,\dots,f_M) \in \xF^{M}$.  Indeed, for each $(f_1,\dots,f_M) \in \xF^{M}$
and  any $(\lambda_{1},  \ldots,  \lambda_{M}) \in  \xR^{M}$, $\bar{f}  \equiv
\sum_{m=1}^{M} \lambda_{m}  f_{m}$ is measurable and  $P_{0} \bar{f}^{2}/h$ is
finite   hence,   by   Lemma~\ref{lem:singlef},  $\sum_{m=1}^{M}   \lambda_{m}
\pG{f_{m}} = \pG(\bar{f})$ converges  in law to $\lG(\bar{f}) = \sum_{m=1}^{M}
\lambda_{m} \pG{f_{m}}$.  In  addition, {\bf A1} implies that  the diameter of
$\xF$  wrt $\|\cdot\|_{2,P_{0}}$  is  finite.  
Therefore,  by  \cite[][Theorems 1.5.4  and  1.5.7]{van1996weak},  if for  all
$\alpha,\beta>0$, there exists $\delta>0$ such that
\begin{equation}
  \label{eq:equicontinuity}
  \limsup_{N\to\infty}    \prob{\sup_{f,f'    :    \|f-f'\|_{2,P_{0}}<\delta}
    \left|\pG f - \pG f'\right|>\alpha} \leq \beta,
\end{equation}
then Theorem~\ref{thm:result_sondage} is valid. 

Set arbitrarily $\alpha, \beta,  \delta> 0$ and introduce $\xF_{\delta} \equiv
\{f-f'  : f,f'\in\xF,  \|f  -  f'\|_{2,P_{0}}\le \delta\}$.   It  is shown  in
\cite{joag-dev1983} (see also \cite{barbour1990}) that $\eta_1, \dots, \eta_N$
are {\em  negatively associated} in the  following sense.  For  each $A_1, A_2
\subset  \{1,  \ldots,  N\}$ with  $A_{1}  \cap  A_{2}  = \emptyset$  and  all
(measurable)  $f:\xR^{d_{1}}  \to \xR$  and  $g:\xR^{d_{2}}  \to \xR$  ($d_{1}
\equiv \text{card}(A_{1})$ and $d_{2}  \equiv \text{card}(A_{2})$), if $f$ and
$g$ are increasing in every coordinate, then
\begin{equation*}
  \text{cov} \left(f(\eta_i : i\in A_1),g(\eta_i : i\in A_2)\right) \le 0.
\end{equation*}

Hoeffding's inequality  for negatively associated bounded  random variables in
\cite[][Theorem S1.2]{bertail:hal-00989585}  guarantees that, conditionally on
$O_1, \dots, O_N$, for all $t>0$,
\begin{equation*}
  \prob{|\pG (f)|> t \middle|O_1,\dots,O_N} \le \exp
  \left(-\frac{2t^2}{\rho_N^2(f)}\right). 
\end{equation*}
Therefore,  a   classical  chaining  argument   \cite[][Corollary  2.2.8,  for
instance]{van1996weak})  yields
\begin{equation}
  \label{eq:chainingEmp}
  \esp{\sup_{f,  f'\in \xF_\delta}|\pG  (f) -  \pG (f')|\middle|O_1,\dots,O_N}
  \lesssim\int_0^\infty \sqrt{\log N(\epsilon,\xF_\delta,\rho_N)} d\epsilon.  
\end{equation}
By  {\bf A2},  there  exists  a deterministic  sequence  $\{a_N\}_{N \geq  1}$
tending  to   0  such   that,  for  all   $f,g  \in  \xF$,   $\rho_N(f,g)  \le
(1+a_N)\rho(f,g)$ $P_{0}$-almost  surely. Consequently, for  every $\epsilon >
0$, it holds $P_{0}$-almost surely that
\begin{equation*}
  N(\epsilon,   \xF_\delta,   \rho_N)   \le  N(\epsilon/(1+a_N),   \xF_\delta,
  \|\cdot\|_{2,P_{0}}). 
\end{equation*}
Plugging  the  previous  upper-bound  in  (\ref{eq:chainingEmp}),  taking  the
expectation, using  Markov's inequality  and letting $N$  go to  infinity then
give
\begin{eqnarray*}
  \limsup_{N \to \infty} \prob{\sup_{f,f' \in \xF_{\delta}}\left|\pG (f) - \pG
      (f')\right|   >   \alpha}   &\lesssim&   \alpha^{-1}\int_0^\infty
  \sqrt{\log N(\epsilon,\xF_\delta,\|\cdot\|_{2,P_{0}})} d\epsilon\\
  &\lesssim & \alpha^{-1} J(\delta, \xF, \|\cdot\|_{2,P_{0}}).  
\end{eqnarray*}
By {\bf A1},  it is possible to choose $\delta>0$ small  enough to ensure that
the    above    RHS    expression    is   smaller    then    $\beta$,    hence
(\ref{eq:equicontinuity}) holds.

It only remains to determine the covariance of $\lG$. By adapting the proof of
Lemma \ref{lem:singlef}, it appears that $\text{cov}(\lG (f),\lG (f')) = P_{0}
ff'/h = \Sigma(f,f')$ for all $f,f' \in \xF$.

\section{
  Tailoring the main theorem in the setting of Section~\ref{subsec:binary}}
\label{sec:adapting:thms}


Let us show that {\bf A1$^b$} and {\bf A2$^b$} imply {\bf A1}--{\bf A4} in the
setting  of  Section~\ref{subsec:binary}.   Since  $\xQ^w$ and  $\xG^{w}$  are
uniformly  bounded away from  0 and  1, $t_n$  \eqref{eq:optShift} necessarily
belongs to a deterministic, compact subset $\xT$ of $\xR$. Define
\begin{equation*}
  \widetilde{\xQ}^w \equiv \left\{\expit \left(\logit Q + t \frac{2A-1}{g(A|W)}\right):
    Q\in \xQ^w, g\in \xG^w, t\in \xT\right\} 
\end{equation*}
then
\begin{equation*}
  \xF \equiv  \{D(P) : P\in \xM  \text{ s.t. }  Q_P\in \widetilde{\xQ}^w, g_P
  \in \xG^w\}. 
\end{equation*}
Obviously, $D(P_n^*) \in \xF$  and $\sup_{f \in \xF}\|f\|_{\infty}$ is finite.
Furthermore, because  $\expit$ is a  1-Lipschitz and $\logit$ is  Lipschitz on
any compact  subset of $(0,1)$,  it holds that  $\widetilde{Q}, \widetilde{Q}'
\in  \widetilde{\xQ}^w$   respectively  parametrized   by  $(Q,  g,   t)$  and
$(Q',g',t')$ satisfy
\begin{equation*}
  \|\widetilde{Q}-\widetilde{Q}'\|_{2,P_0}                             \lesssim
  \|Q-Q'\|_{2,P_0}+\|g-g'\|_{2,P_0}+|t-t'|. 
\end{equation*}
Therefore,   the    finiteness   of   $J(1,    \xQ^w,   \|\cdot\|_{2,P_{0}})$,
$J(1,  \xG^w,  \|\cdot\|_{2,P_{0}})$  and  $J(1, \xT,  |\cdot|)$  implies  the
finiteness of  $J(1, \widetilde{\xQ}^w, \|\cdot\|_{2,P_{0}})$.   Moreover, the
separability of  $\xQ^{w}$ and $\xG^{w}$ yields  that $\widetilde{\xQ}^{w}$ is
also separable.

Furthermore, for every $P, P'\in \xM$  such that $D^b(P), D^b(P') \in \xF$, it
holds that
\begin{equation}
  \label{eq:tech:bin}
  \|D^{b} (P) - D^{b} (P')\|_{2,P_{0}} \lesssim \|Q_{P} - Q_{P'}\|_{2,P_{0}}
  + \|g_{P} - g_{P'}\|_{2,P_{0}} + |\Psi^{b} (P) - \Psi^{b} (P')|.
\end{equation}
We will prove this at the end of the section.
By   (\ref{eq:tech:bin}),  the   separability  of   $\widetilde{\xQ}^{w}$  and
$\xG^{w}$ implies that of~$\xF$.
In             addition,             the             finiteness             of
$J(1,                \widetilde{\xQ}^{w},               \|\cdot\|_{2,P_{0}})$,
$J(1,   \xG^{w},    \|\cdot\|_{2,P_{0}})$,   $J(1,   [0,1],    |\cdot|)$   and
(\ref{eq:tech:bin}) imply that $J(1, \xF, \|\cdot\|_{2,P_{0}})$ is finite.  We
prove likewise  based on (\ref{eq:tech:bin})  that $\xF$ has a  finite uniform
entropy   integral    because   $\xQ^{w}$   and   $\xG^{w}$    do.    Finally,
\eqref{eq:verifRasympt} and  {\bf A2$^b$} imply {\bf  A3} (by Cauchy-Schwarz's
inequality) and {\bf A4}.

\begin{proof}[Proof of (\ref{eq:tech:bin})]
  For  any  $P\in   \xM$,  denote  $q_P(W)=Q_P(1,W)-Q_P(0,W)$.   Set  $P,P'\in
  \xM^b$. It holds that
  \begin{align*}
    \|D^b_1(P)-D^b_1(P')\|_{2,P_0}    &     \le    \|q_P-q_{P'}\|_{2,P_0}    +
    |\Psi^b(P)-\Psi^b(P')|. 
  \end{align*}
  Moreover, 
  \begin{align*}
    \|D^b_2(P)-D^b_2(P')\|_{2,P_0}                     &                     =
    \left\|(Y-q_P(W))\frac{2A-1}{g_P}-(Y-q_{P'}(W))\frac{2A-1}{g_{P'}}\right\|_{2,P_0}\nonumber\\ 
    &                                                                       \le
    \left\|(Y-q_P(W))(2A-1)\left(\frac{1}{g_P}-\frac{1}{g_{P'}}\right)\right\|_{2,P_0}+\left\|(q_P-q_{P'})\frac{2A-1}{g_{P'}(W)}\right\|_{2,P_0}\nonumber\\ 
    & \lesssim \|g_P-g_{P'}\|_{2,P_0} + \|q_P-q_{P'}\|_{2,P_0},
  \end{align*}
  where   the  last   inequality  relies   on  the   uniform   boundedness  of
  $(Y-q_P(W))(2A-1)$    and   $g_P^{-1}$.     The    result   follows    since
  $\|q_P-q_{P'}\|_{2,P_0} \le 2 \|Q-Q'\|_{2,P_0}$.
\end{proof} 

The same kind of arguments allow  to verify that {\bf A1$^c$} and {\bf A2$^c$}
also imply {\bf A1}--{\bf A4}.

\bibliographystyle{plainnat}

\bibliography{biblio}

\end{document}